\documentclass[11pt]{article}
\usepackage{mathrsfs}
\usepackage[T1]{fontenc}
\usepackage{latexsym,amssymb,amsmath,amsfonts,amsthm}
\usepackage{graphicx,psfrag,wrapfig}
\newcommand{\R}{{\mathbb R}}

\newcommand{\N}{{\mathbb N}}

\newcommand{\ITP}{\textbf{(ITP) }}
\newcommand{\BVP}{\textbf{(BVP) }}
\newcommand{\SP}{\textbf{(SP) }}

\newcommand{\no}{\nonumber}
\newcommand{\be}{\begin{eqnarray}}
\newcommand{\ben}{\begin{eqnarray*}}
\newcommand{\en}{\end{eqnarray}}
\newcommand{\enn}{\end{eqnarray*}}
\newcommand{\ba}{\backslash}
\newcommand{\pa}{\partial}

\newcommand{\ov}{\overline}

\newcommand{\g}{\gamma}
\newcommand{\G}{\Gamma}

\newcommand{\Om}{\Omega}

\newcommand{\wi}{\widetilde}

\newcommand{\Int}{\int\limits}
\newcommand{\re}{\text{Re}\,}
\newcommand{\im}{\text{Im}\,}
\newcommand{\dd}{\,\text{d}}
\newtheorem{theorem}{Theorem}[section]
\newtheorem{lemma}[theorem]{Lemma}
\newtheorem{corollary}[theorem]{Corollary}

\newtheorem{remark}[theorem]{Remark}

\begin{document}
\renewcommand{\theequation}{\arabic{section}.\arabic{equation}}
\title{Direct and inverse scattering problems by an unbounded rough interface with buried obstacles}
%
\author{Yulong Lu\thanks{Mathematics Institute, University of Warwick, Coventry, CV4 7AL, UK
({\sf yulong.lu@warwick.ac.uk}). }
\and
Bo Zhang\thanks{LSEC and Academy of Mathematics and Systems Sciences, Chinese Academy of Sciences,
Beijing, 100190, China and School of Mathematical Sciences, University of Chinese Academy of Sciences, 
Beijing 100049, China ({\sf b.zhang@amt.ac.cn}).}
}
\date{} 
\maketitle

\begin{abstract}
In this paper, we consider the direct and inverse problem of scattering of time-harmonic
waves by an unbounded rough interface with a buried impenetrable obstacle. We first study
the well-posedness of the direct problem with a local source by the variational method; the
well-posedness result is then extended to scattering problems associated with point source waves
(PSWs) and hyper-singular point source waves (HSPSWs). For incident PSW or HSPSW waves,
the corresponding total field admits a uniformly bounded estimate in any compact subset
far away from the source position. Moreover, we show that the scattered field due to HSPSWs
can be approximated by the scattered fields due to PSWs. With these properties and a
novel reciprocity relation of the total field, we prove that both the rough surface and the buried
obstacle can be uniquely determined by the scattered near-field data measured only on a line
segment above the rough surface. The proof substantially relies upon constructing a well-posed
interior transmission problem for the Helmholtz equation.

\vspace{0.2in}

\vspace{.1in}
\end{abstract}

\section{Introduction}\label{sec1}
\setcounter{equation}{0}

This paper is concerned with the problem of scattering of time-harmonic waves from an unbounded
rough interface with a buried impenetrable obstacle in two dimensions.
This model problem has extensive applications in physics and engineering, such as
ocean exploration by sonar and remote sensing by synthetic aperture radar (SAR).
The {\em unbounded rough interface} is assumed to be a non-local perturbation of
an infinite plane such that the interface lies within a finite distance of the original plane.
We assume further that the whole space is separated by the unbounded rough interface
with the medium above and below the rough interface being both homogeneous and isotropic.
Many work has been done on the numerical approximation and computation for rough surface
scattering problems (see, e.g. \cite{DeS,Li16,Ogilvy,AG,Chew} and the references quoted therein).
The mathematical theory of rough surface scattering problems has also been studied by many authors
using integral equation methods (see, e.g. \cite{SN4,ZhangRS,CZ1999,SN5,NAC,SN1,SN2})  
or by the variational approach (see, e.g. \cite{SN6,SN3,HL2011,Hu15,AS,Li}).
It should be mentioned that the variational approach first proposed in \cite{SN3} for the rough surface scattering 
can be applied to study the well-posedness of the scattering problems by unbounded Lipschitz surfaces in both two 
and three dimensions. This approach can also give an a priori estimate
of the solution in terms of the data with an explicit dependence on the wave number.

In this paper, we consider the direct scattering problem modeled by the Helmholtz equation
$\Delta u + k^2u = g$ in $\R^2$ with the wave number $k^2=k^2_1$ above the rough interface
and $k^2=k^2_2$ below it. And the total field $u$ satisfies transmission conditions
on the rough interface and boundary conditions on the buried impenetrable obstacle $D$.
The model includes the scattering excited by a local source when $g\in L^2(\R^2)$
with a compact support and a point source wave or a singular point source when $g$ denotes
a general distribution. Figure \ref{fig1} presents the geometrical setting of the scattering problem.
To accomplish the scattering problem, a radiation condition at infinity is required.
Due to the unbounded rough surface, the Sommerfeld radiation condition is no longer valid.
We require that the solution above the rough interface and below the buried obstacle
can be represented in an integral form as a superposition of upward (downward) propagating
and evanescent plane waves. This radiation condition is equivalent to the upward propagating radiation 
condition first proposed by Chandler-Wilde and Zhang in \cite{SN5} for the two-dimensional case.

Related work on the direct scattering problem can be found in \cite{SN6,SN3,HL2011,Hu15,AS,Li}.
These papers employed the variational method to study the acoustic scattering from sound-soft or 
sound-hard rough surfaces or penetrable rough layers and the electromagnetic scattering from rough layers 
with an absorbing medium. Different from these work, this paper focuses on
the wave scattering from an unbounded rough interface with a buried impenetrable obstacle. 
The existence of an obstacle in the model will make the analysis much more complicated.
In particular, we can not obtain a priori estimates in terms of the data in case of non-absorbing medium
because of sign-changing terms on the boundary of the obstacle.
However, the a priori estimate can be established under the condition that the medium below
the rough interface is absorbing. This condition fits well with certain engineering applications,
such as underground remote sensing since the soil is in fact energy-absorbing.
In the non-absorbing case, the variational formation is reduced into an operator equation
with the operator being Fredholm with index zero. Thus, the existence of solutions follows from the
uniqueness of solutions. In particular, our scattering problem is well-posed in the case when the
obstacle is partially coated in a non-absorbing medium. The existence of solutions
to the scattering problem due to PSWs and HSPSWs was studied already in a different setting in \cite{SN6}.
However, we have the following key observations. First, we show that the total field is uniformly
bounded with respect to the source positions in any compact set far away from the source position 
(see Theorem \ref{thm4_1}). This uniform bound is useful for constructing a well-posed interior transmission
problem that will be used to prove the uniqueness result for the inverse problem. Moreover, we show that
the scattered field due to HSPSWs can be approximated by the scattered field due to PSWs (see Theorem \ref{thm3}).
Since we will mainly employ the singularity of HSPSWs in the study of the inverse scattering problem, 
this approximation result makes it possible to use the scattered field induced by PSWs instead.

As for uniqueness results for inverse scattering problems, there exist a vast literature on the case of 
bounded obstacle scattering problems (see, e.g. \cite{Isakov,Kirsch3}).
Moreover, inverse scattering from a multilayered background medium is also studied (see \cite{HE,LZ1}). 
However, the method used in \cite{HE,LZ1} only works for the case when the transmission constant $\lambda\neq 1$, 
and the method used in \cite{LZ1} also relies heavily on a priori estimates of the scattering solution on the interface 
between a layered medium which is hard to be established in the case of rough surface scattering problems.
There are also numerous uniqueness results on inverse scattering problems on periodic structures, 
which can be viewed as a special case of rough surfaces (see, e.g. \cite{Bao,Kirsch2,Kirsch1,YZ1,YZ2} 
and the references quoted there). 
Recently, the scattering problems have also been studied in \cite{LZ2,Delbary,Griesmaier} from an obstacle 
in a two-layered background medium with a planar interface.

There are only few uniqueness results on inverse rough surface scattering problems. In \cite{SN7}, 
Chandler-Wilde and Ross proved that a sound-soft rough surface in
a lossy medium can be uniquely determined by the scattered field associated with only one incident plane wave. 
Hu \cite{Hu} proved that sound-soft rough surfaces and rough layers (with transmission constant $\lambda\neq 1$) 
can be uniquely recovered from the scattered field due to PSWs.

Recently, Yang, Zhang and Zhang \cite{YZZ} proposed a new method to prove uniqueness of inverse
scattering from penetrable obstacles including the case when the transmission constant $\lambda$ = 1.
The main idea is based on constructing a well-posed interior transmission problem on a small domain.
Precisely, suppose that there are two obstacles which produce the same scattered data. 
One constructs a local well-posed interior transmission problem with the boundary data given by 
the scattered field corresponding to point sources and different obstacles. The scattered field 
corresponding to one obstacle can be shown uniformly bounded as the source position approaches the boundary 
of the other obstacle. Then one uses this fact and the well-posedness of the interior transmission problem 
to get the contradiction that the $H^1$-norm (or $L^2$-norm) of the point sources (or the hyper singular point sources)
are uniformly bounded. An important feature of this idea is that the interior transmission
problem is constructed locally on a small domain. This motivates us to adapt the similar idea to prove uniqueness
results in rough surface scattering problems. We note that in the proof of uniqueness results in
\cite{YZZ}, a denseness result (Theorem 5.5 in \cite{Colton}) of incident plane waves
is used, which can not be generalized to incident point source waves in rough surface scattering problems. 
In our proof, the denseness result is replaced by the approximation property of the
scattered field due to PSWs and HPSWs.

This paper is organized as follows. In Sections \ref{sec2} and \ref{sec3}, we formulate the
boundary value problem modeling the direct scattering problem with a local source $g$ and
give its equivalent variational formulation. Then we study the solvability of the variational
formation in two different cases according to whether the medium below the rough interface
is lossy. In Section \ref{sec4}, we show that similar results also hold for incident PSWs and HSPSWs waves. 
Moreover, we prove two important results about the scattered field, that is, the uniform boundedness of 
the total field with respect to the source positions and the approximation property about the scattered field.
In section \ref{sec5}, we first state some results on interior transmission problems and then prove the
uniqueness result of the inverse scattering problem based on the interior transmission problem and a
novel reciprocity relation.

\section{The direct problem and its variational formulation}\label{sec2}
\setcounter{equation}{0}

In this section, we present the direct problem and its equivalent
variational formulation. To this end, we need some notations. For $h\in\R$,
let $\G_h=\{x=(x_1,x_2)\in\R^{2}\,|\,x_2=h\}$ and denote $U_h^{\pm}=\{x\in\R^2\,|\,x_2\gtrless h\}$.
For a given bounded function $f\in C^2(\R)$, we define $f_{-}:=\text{inf}_{x\in\R}f(x)>0,\;
f_{+}:=\text{sup}_{x\in\R}f(x)<+\infty$. Then the rough interface is defined by
$\G :=\{(x_1,f(x_1))\,|\,x_1\in\R\}$.
Denote by $D$ the buried impenetrable obstacle with boundary $\pa D\in C^2$,
and assume that $D$ is below the rough interface, this is,
$\text{dist}(\overline{D},\overline{U_{f_{-}}^+})>0$.
For simplicity, we assume that $D\subset U^{-}_0$. Assume further that the buried obstacle is partially coated
by a thin dielectric layer so that $\pa D=\overline{\G}_1\cup\overline{\G}_2$,
where $\G_1 $ and $\G_2$ are two disjoint open subsets of $\pa D$.
Denote by $\G_1$ the coated part with an impedance function $\beta(x)$ and by $\G_2$ the uncoated part.
In particular, the obstacle is sound-soft if $\G_1 = \varnothing$, and
fully coated obstacle (an impedance obstacle) corresponds to the case when $\G_2 = \varnothing$.
Note also that the obstacle becomes sound-hard when the impedance function $\beta(x)$ vanishes on $\pa D$.
Denote by $\Om_1$ the region above $\G$, and by $\Om_2$ the region below $\G$ and outside $D$.
We also define $\Om_{1,h}:=\{x=(x_1,x_2)\in\Om_1|x_2 < h\}$, $\Om_{2,h}:=\{x=(x_1,x_2)\in\Om_2|x_2>-h\}$.
For simplicity, let $h>f_+$ satisfy that $D\subset U^{+}_{-h}$, and define $\Om_h:=\{x\in\R^2|-h<x_2<h\}\setminus\overline{D}$.
For $M >0$, denote $\Om_{i,h}(M):=\{x\in\Om_{i,h}|-M\leq x_1\leq M\}$
and $\g_i(\pm M):=\{x\in\Om_{i,h} | \ |x_1|=\pm M\}, i = 1,2$.
Let $\nu(x)$ be the unit normal vector at $x\in \G$ pointing into $\Om_1$
or at $x\in \pa D$ pointing out of D. For $\varepsilon>0$, and $y\in\R^2$,
denote by $B_{\varepsilon}(y)$ the ball centered at $y$ with radius $\varepsilon$.
We first consider the scattering problem with a local source $g\in L^2(\R^2)$
compactly supported in $\Om_h$. The cases with incident waves PSWs and HSPSWs will be considered
in Section \ref{sec4}.

\begin{figure}
\centering
\includegraphics[scale=0.5]{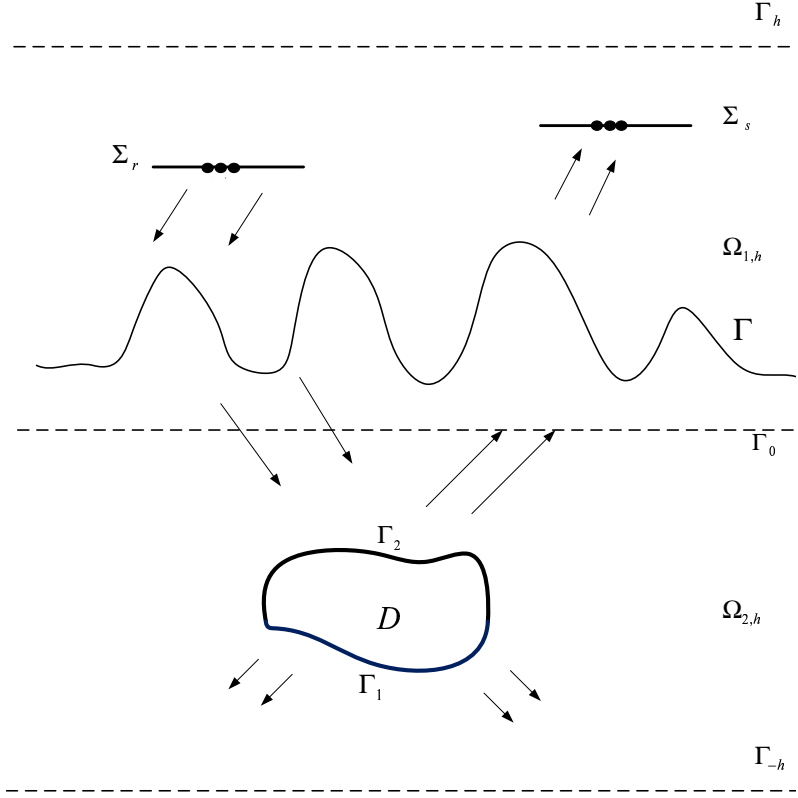}
\caption{Scattering from an unbounded rough interface with impenetrable obstacles.}\label{fig1}
\end{figure}

We are now ready to formulate the scattering problem. Assume that $\Om_1$ and $\Om_2$ are
filled with two isotropic homogenous materials denoted by the wave numbers $k_1$ and $k_2$
respectively satisfying
\begin{equation}\label{eq2_1}
        \begin{aligned}
 k^2_1> \re(k^2_2)> 0 &  \;\; \text{or}\;\;0<k^2_1< \re(k^2_2)\\
 & \im(k^2_2)\geq 0
\end{aligned}
\end{equation}
This means that the medium above the rough interface is
non-absorbing and that below the interface it may be absorbing. The condition \eqref{eq2_1} is
usually termed as {\em non-trap} condition because it
ensures the uniqueness of the scattering problem.
We only consider the case $k^2_1> \re(k^2_2)$; the other one case can be dealt with similarly (see Remark \ref{rem3}).
The total field $u$ due to the source $g$ satisfies the Helmholtz equations
\be\label{eq1}
\Delta u + k^2 u = g\quad\text{in}\;\;\R^2\backslash D
\en
where $k^2(x):= k^2_1$ for $x\in\Om_1$ and $k^2(x):=k^2_2$ for $x\in\Om_2$.
On the rough interface, the total field $u$ satisfies the transmission condition
\be
\label{eq2}
u^{+}=u^{-},\;\frac{\pa u^{+}}{\pa\nu}=\frac{\pa u^{-}}{\pa\nu}\quad\text{on}\;\;\G,
\en
where $u^{+},{\pa u^{+}}/{\pa\nu}$ (resp. $u^{-}, {\pa u^{-}}/{\pa\nu}$) denote the limits on $\G$
from the above (resp. below). This implies that the field and its normal derivatives are continuous across the interface.
On the boundary of the buried obstacle $\pa D$, the field $u$ satisfies a mixed boundary condition
\be\label{eq3}
\frac{\pa u}{\pa\nu}u+i\beta u=0\quad\text{on}\;\;\G_1,\quad u=0\quad\text{on}\;\;\G_2
\en
with $\beta\geq 0,\;\beta\in C(\G_1)$ representing the physical property of the obstacle.
We use the condition $\mathcal{B}(u)=0$ to denote the boundary condition \eqref{eq3}.

Since $\Om_1$ and $\Om_2$ are unbounded, radiation conditions at infinity must be
imposed on the scattered and transmitted field. It is worth to note that the standard Sommerfeld radiation condition
is not appropriate for rough surface scattering problems.
Similar to \cite{SN3}, the scattered field is required to be represented in an integral form
as a superposition of upward (resp. downward) propagating and
evanescent plane waves in $U_h^{+}$ (resp. $U_{-h}^{-}$).

For $\phi\in L^2(\R)$, define its Fourier transform by
\ben
\hat{\phi}(\xi):=\mathcal{F}\phi(\xi)=\frac{1}{\sqrt{2\pi}}
\int_\R\exp(-ix_1\cdot\xi)\phi(x_1)\dd x_1,\quad\xi\in\R.
\enn
We require $u$ to satisfy the angular spectrum representation:
\be\label{eq4}
u(x)=\frac{1}{\sqrt{2\pi}}\Int_\R{}\text{exp}(i[(x_2-h)\sqrt{k_1^2-\xi^2}+x_1\cdot\xi])
    \mathcal{F}(u|_{\G_{h}}) (\xi) \dd\xi, \quad x\in U^+_{h},\\ \label{eq5}
u(x)=\frac{1}{\sqrt{2\pi}}\Int_\R{}\text{exp}(i[-(x_2+h)\sqrt{k_2^2-\xi^2}+x_1\cdot\xi])
    \mathcal{F}(u|_{\G_{-h}}) (\xi) \dd\xi, \quad x\in U^-_{-h},
\en
where $h>f_+$, $u|_{\G_{h}}\in L^2(\G_{h})$, and the square root in the expression takes
the negative imaginary axis as the branch cut in the complex plane, that is for
$z\in\mathbb{C}, z = z_1 + iz_2,z_1,z_2\in\R$, we have
\be\label{eq6}
\sqrt{z} = \text{sgn}(z_2)\sqrt{\frac{|z| + z_1}{2}} + i \sqrt{\frac{|z| - z_1}{2}}
\en
Define $V_h:=\{u|u\in H^1(\Om_h),\;\;u=0\;\;\text{on}\;\;\G_2\}$.
The inner product and norm in $V_h$ are the same as the function space $H^1(\Om_h)$. Then the direct
scattering problem can be stated as the following boundary value problem.

{\bf Boundary Value Problem (BVP):} Given a source $g\in L^2(\R^2)$ compactly supported in $\Om_h$,
find $u$ such that $u \in V_h$ satisfying \eqref{eq1}-\eqref{eq3} and the radiation conditions \eqref{eq4} and \eqref{eq5}.

For $s \in\R$, define $H^s(\G_h)$  as the completion of $C_0^{\infty}(\G_h)$ in the following norm
\be\label{eq7}
\| \phi\|^2_{H^s(\G_h)} : = \Int_{\R} (1 + \xi^2)^s |\hat{\phi}(\xi)|^2 \dd\xi
\en
Introduce Dirichlet-to-Neumann (DtN) operators $T_1$ on $\G_h$ and $T_2$ on $\G_{-h}$
\ben
(T_1\phi)(x_1)=\frac{i}{\sqrt{2\pi}}\Int_{\R}\sqrt{k_1^2-\xi^2}\text{exp}(ix_1\cdot\xi)\hat{\phi}(\xi)\dd\xi,\quad x\in \G_h\\
(T_2\phi)(x_1)=\frac{i}{\sqrt{2\pi}}\Int_{\R}\sqrt{k_2^2-\xi^2}\text{exp}(ix_1\cdot\xi)\hat{\phi}(\xi)\dd\xi,\quad x\in \G_{-h}
\enn
The next lemma collects some properties of the DtN operators.

\begin{lemma}\label{lem0}
(i) $T_1: H^{1/2}(\G_h)\rightarrow H^{-1/2}(\G_h)$ and
$T_2: H^{1/2}(\G_{-h})\rightarrow H^{-1/2}(\G_{-h})$ are bounded linear operators.

(ii) For $\phi\in H^{1/2}(\G_h)$ and $\psi\in H^{1/2}(\G_{-h})$, we have
\be\label{eq8}
\re(\Int_{\G_h}\overline{\phi}T_1\phi \dd s)\leq 0,\quad
  \im(\Int_{\G_h}\overline{\phi}T_1\phi \dd s)\geq0,\\ \label{eq9}
\re(\Int_{\G_{-h}}\overline{\psi}T_2\psi \dd s)\leq0,\quad
 \im(\Int_{\G_{-h}}\overline{\psi}T_2\psi \dd s)\geq 0
\en

(iii) For $\phi_j\in H^{1/2}(\G_h)$ and $\psi_j\in H^{1/2}(\G_{-h})$, $j = 1,2$, we have
\be\label{eq10}
\Int_{\G_{h}}\phi_1T_1\phi_2 \dd s=\Int_{\G_{h}}\phi_2T_1\phi_1 \dd s,\quad
 \Int_{\G_{-h}}\psi_1T_2\psi_2 \dd s = \Int_{\G_{-h}}\psi_2T_2\psi_1 \dd s.
\en
\end{lemma}

\begin{proof}
(i) From the definition of $T_1$ and $T_2$ and by \eqref{eq7}, we have
\ben
\|T_1\phi\|^2_{H^{-1/2}(\G_h)}=\Int_{\R}(1+\xi^2)^{-1/2}|\widehat{T_1\phi}(\xi)|^2 \dd\xi
=\Int_{\R}(1+\xi^2)^{-1/2}|\sqrt{k_1^2-\xi^2}|^2|\widehat{\phi}(\xi)|^2\dd\xi
\enn
Noting that $|\sqrt{k_1^2-\xi^2}|^2\leq C(k_1)(1+\xi^2)^{1/2}$, then
\ben
 \|T_1\phi\|^2_{H^{-1/2}(\G_h)} \leq C(k_1) \|\phi\|^2_{H^{1/2}(\G_h)}
\enn
Similarly, we have
\ben
\|T_1\psi\|^2_{H^{-1/2}(\G_{-h})} \leq C(k_2) \|\psi\|^2_{H^{1/2}(\G_{-h})}
\enn

(ii)  \eqref{eq8} is proved in \cite{SN3}. For $T_2$, we have
\ben
\Int_{\G_{-h}}\overline{\psi}T_2\psi \dd s=\Int_{\R}i\sqrt{k_2^2-\xi^2}|\mathcal{F}\psi(\xi)|^2 \dd \xi.
\enn
By \eqref{eq6} and note that $\im k_2 \geq 0$,
\ben
\re (\sqrt{k_2^2 - \xi^2}) \geq 0, \im (\sqrt{k_2^2 - \xi^2}) \geq 0
\enn
which implies \eqref{eq9}.

(iii) It is proved in \cite{SN3} (see Lemma 3.2 therein).
\end{proof}

\begin{lemma}\label{lem2}
(i) If $u$ satisfies \eqref{eq4} with $u|_{\G_{h}}\in H^{1/2}(\G_{h})$,
then $u\in H^1(U_{h}^+\backslash U_{a}^+)\cap C^2(U_{h}^+)$ for every $a>h$,
\ben
\Delta u + k_1^2 u = 0 \quad \text{in} \quad U^+_{h},
\enn
$\g_+ u = u|_{\G_{h}}$, and
\ben
\Int_{\G_{h}}\overline{v}T_1\g^+u \dd s+k_1^2\Int_{U_{h}^+}u\overline{v}\dd x
-\Int_{U_{h}^+}\nabla u\cdot\nabla\overline{v} \dd x = 0,\quad v\in C_0^{\infty}(\R^2)
\enn
where $\g^+$ is the trace operator from $V_h$ to $H^{1/2}(\G_h)$.
Further, for all $a > h$, the restriction of $u$ and $\nabla u$ to $\G_{a}$ lies in $L^2(\G_a)$ and
\be \label{eq11}
\Int_{\G_a} (|\frac{\pa u}{\pa x_2}|^2 - |\frac{\pa u}{\pa x_1}|^2 + k_1^2 |u|^2) \dd s
\leq 2k_1\im(\Int_{\G_a}\overline{u} \frac{\pa u}{\pa x_2}\dd s).
\en
Moreover, \eqref{eq4} holds with $h$ replaced by $a$ for all $a>h.$ 

(ii) If $u$ satisfies \eqref{eq5} with $u|_{\G_{-h}}\in H^{1/2}(\G_{h})$,
then $u\in H^1(U_{-h}^{-}\backslash U_{a}^{-})\cap C^2(U_{-h}^{-})$, for every $a < -h$,
\ben
\Delta u + k_2^2 u = 0 \quad \text{in} \quad U^{-}_{-h},
\enn
$\g_{-} u = u|_{\G_{-h}}$, and
\ben
\Int_{\G_{-h}} \overline{v}T_2 \g^{-} u \dd s + k_2^2\Int_{U_{-h}^{-}}u\overline{v} \dd x
- \Int_{U_{-h}^{-}}\nabla u\cdot\nabla\overline{v} \dd x = 0,\quad v\in C_0^{\infty}(\R^2)
\enn
where $\g^{-}$ is the trace operator from $V_h$ to $H^{1/2}(\G_{-h})$.
Further, for all $a < -h$, the restriction of $u$ and $\nabla u$ to $\G_{a}$ lies in $L^2(\G_a)$ and
\be\label{eq11_1}
\Int_{\G_a} \left(|\frac{\pa u}{\pa x_2}|^2 - |\frac{\pa u}{\pa x_1}|^2 + \re(k_2^2) |u|^2\right) \dd s
\leq \left(2(\re(k^2_2))^{1/2} + \sqrt{2}(\im(k^2_2))^{1/2}\right)
\im(\Int_{\G_a}\overline{u} \frac{\pa u}{\pa x_2} \dd s).
\en
Moreover, \eqref{eq4} holds with $-h$ replaced by $a$ for all $a<-h.$ 
\end{lemma}

\begin{proof}
The lemma can be proved similarly as in \cite{SN3} and \cite{AS}.
\end{proof}

Multiplying \eqref{eq1} by $v\in V_h$ and integrating by parts, we can easily get the following
equivalent variational formulation of the problem \BVP.

Find $u\in V_h$ such that
\be\label{eq12}
a(u,v) = -\Int_{\Om_h}g\overline{v}\, \dd x \ \text{ for all } v\in V_h.
\en
where the sesquilinear form $a(u,v)$ is defined as
\be\label{eqlf}
a(u,v) := \Int_{\Om_h}(\nabla u \nabla \overline{v} - k^2u \overline{v})\dd x
-\Int_{\G_h}\overline{v}T_1 u \dd s-\Int_{\G_{-h}}\overline{v}T_2 u \dd s
-i\Int_{\G_1}\beta u\overline{v}\dd s.
\en
The problem \BVP and variational formulation \eqref{eq12} are equivalent in the following sense:
Given $u$ satisfying \BVP, $u|_{\Om_h}$ is a solution of \eqref{eq12}.
Conversely, if $u$ is a solution of \eqref{eq12}, it is easy to see that $u$ satisfies
the transmission condition \eqref{eq2} on $\G$ and the boundary condition \eqref{eq3} on $\pa D$.
From Lemma \ref{lem2}, we can expand $u$ to $\R^2$ by \eqref{eq4} and \eqref{eq5} with
continuous traces on $\G_h$ and $\G_{-h}$. Moreover, $u$ satisfies $\Delta u + k^2u = g$
in the distribution sense, with $g$ extended to be zero outside $\Om_h$.
Consequently, $u$ also satisfies the problem \BVP.

From the definition of $a(\cdot, \cdot)$, the boundedness of $T_1$ and $T_2$ and the fact
that $\beta \in C(\G_1)$, the sesquilinear form $a(\cdot,\cdot)$ defined by \eqref{eqlf} is
bounded, that is,
\ben 
|a(u,v)| \leq C \|u\|_{V_h} \|v\|_{V_h}, \  u, v \in V_h
\enn
By the Riesz representation theorem there exists a bounded linear operator
$\mathcal{A}_k: V_h \rightarrow V_h^{\ast}$ such that
\ben 
\langle \mathcal{A}_k u, v \rangle_{V_h} = a(u,v), \  u, v \in V_h
\enn
where $V_h^{\ast}$ denotes the dual space of $V_h$ and $\langle \cdot,\cdot \rangle_{V_h}$ is the dual pair
between $V_h^{\ast}$ and $V_h$. Note that $\mathcal{A}_k$ depends on $k$
since $a(\cdot,\cdot)$ does.
Therefore, the variational formulation \eqref{eq12} can also be simplified as the following
operator equation
\be\label{eq_oe}
\mathcal{A}_k u = \mathcal{G}
\en
where $\mathcal{G} \in V_h^{\ast}$ is defined by
$\mathcal{G}(v):=-\Int_{\Om_{h}}g\overline{v}\dd x,\;v\in V_h$,
with the local source $g\in L^2(\R^2)$ and
$\|\mathcal{G}\|_{V_h^{\ast}} \leq \|g\|_{L^2(\R^2)}.$

\section{Well-posedness of the variational formulation}\label{sec3}
\setcounter{equation}{0}

In this section, we prove the well-posedness of the variational problem \eqref{eq12}
and the well-posedness of the problem \BVP follows subsequently.
The former mainly depends on the generalized Lax-Milgram theory of Babu$\check{s}$ka
(Theorem 2.15 in \cite{Ih}). Thus we reformulate the variational problem into a more general
problem in the framework of functional analysis: given $\mathcal{G}\in V_h^{\ast}$ find $u\in V_h$
such that $\mathcal{A}_k u = \mathcal{G}$ or
\be\label{eq_3_1}
a(u,v) = \mathcal{G}(v),\;\;\forall v\in V_h
\en

\begin{theorem}\label{thm1}
(Generalized Lax-Milgram Theorem) Let $H$ be a Hilbert space with norm and inner product
given by $\|\cdot\|$ and $(\cdot,\cdot)$ respectively.
Suppose that $a:H\times H\rightarrow\mathbb{C}$ is a bounded sesquilinear form such that there holds the inf-sup condition
\be \label{eq_3_2}
\gamma := \inf_{0\neq u \in H} \sup_{0 \neq v \in H}\frac{|a(u,v)|}{\|u\|\|v\|} > 0
\en
and the transposed inf-sup condition
\ben
\sup_{0 \neq u \in H}\frac{|a(u,v)|}{\|u\|} > 0.
\enn
Then for each $\mathcal{G} \in H^{\ast}$ there exists a unique solution $u \in H$ such that
\ben
a(u,v) = \mathcal{G}(v) \ \forall v\in H, \ with \ \|u\| \leq \gamma^{-1}\|\mathcal{G}\|_{H^{\ast}}.
\enn
\end{theorem}

The inf-sup condition, which is the key requirement in the Generalized Lax-Milgram Theorem,
can be verified by the following lemma \cite{Ih}.

\begin{lemma} \label{lem4}
Suppose there exists $C > 0$ such that for all $u\in V_h$ and $\mathcal{G}\in V_h^{\ast}$
satisfying \eqref{eq_3_1} it holds that
\be \label{eq_3_3}
\|u\|_{V_h} \leq C\|\mathcal{G}\|_{V_h^{\ast}}.
\en
Then the inf-sup condition \eqref{eq_3_2} holds with $\gamma\geq C^{-1}$ and $H =V_h$.
\end{lemma}

In order to obtain the a priori estimate \eqref{eq_3_3}, we consider two cases
depending on whether or not the medium below the rough interface is absorbing.

\subsection{Case 1: $0 <\re(k_2^2)< k^2_1, \im(k_2^2) > 0$}

It is shown in Lemma 4.5 of \cite{SN3} that the a priori estimate \eqref{eq_3_3} for the solution
of \eqref{eq_3_1} can be obtained by the a priori estimate for the solution of \eqref{eq12} with
$g\in L^2(\Om_h)$. We now prove the later estimate by the Rellich identity technique,
which was used in \cite{SN3,AS}.

\begin{lemma}\label{lem3_4}
For the given $g\in L^2(\Om_h)$, let $u \in V_h$ satisfy the problem
\be\label{eq_3_4}
a(u,v)=-(g,v)\;\;\text{for all}\;\;v\in V_h.
\en
Then
\be\label{eq13}
\|u\|_{V_h} \leq C \|g\|_{L^2(\Om_h)}
\en
\end{lemma}

\begin{proof}
Taking the real and imaginary part of \eqref{lem3_4} with $v=u$ leads to the equation:
\be\label{eq15}
&&\Int_{\Om_{h}}(|\nabla u|^2 - k^2 |u|^2)\dd x - \re\Int_{\G_h}\overline{u}T_1u\dd s
    -\re \Int_{\G_{-h}}\overline{u}T_2u \dd s = -\re\Int_{\Om_h}g\overline{u}\dd x,\\ \label{eq16}
&&\im(k_2^2)\Int_{\Om_{2,h}}|u|^2 \dd x + \im\Int_{\G_h}\overline{u}T_1u \dd s
  + \im\Int_{\G_{-h}}\overline{u}T_2u\dd s+\Int_{\G_1}\beta|u|^2 \dd s =\im\Int_{\Om_h}g\overline{u}\dd x
\en
By the standard elliptic regularity estimate \cite{Gilbarg} and since
$g\in L^2(\Om_h),\G \in C^2$ and $\pa D\in C^2$, we have $u \in H^2_{\text{loc}}(\Om_h)$.
For $A >0$, let $\varphi_A(\cdot)\in C_0^{\infty}(\R)$ be a smooth cut-off
function such that $0\leq \varphi_A(r)\leq 1,\varphi_A(r)=1\;\;\text{if}\;r\leq A,\varphi_A=0\;\text{if}\;r\geq A+1,$
and $\|\varphi^{\prime}_A\|_{L^{\infty}(\R)}<\infty$. Applying the Green's theorem to $u$
and $\varphi_A(|x_1|)(x_2+h)\pa\overline{u}/\pa x_2$ in $\Om_{1,h}(A)$ and $\Om_{2,h}(A)$ and
letting $A\rightarrow \infty$, we have
\be\no
&& 2h\Int_{\G_h}\left(|\frac{\pa u}{\pa x_2}|^2-|\frac{\pa u}{\pa x_1}|^2+k_1^2|u|^2\right)\dd s\\ \no
&&\qquad\qquad+\int_{\G}(x_2+h)\left(\left(|\nabla u^+|^2-k_1^2|u^+|^2\right)\nu_2
   -2\re\left(\frac{\pa u^+}{\pa\nu}\frac{\pa\ov{u}^+}{\pa x_2}\right)\right)\dd s\\ \label{eq17}
&&\qquad\qquad+\Int_{\Om_{1,h}}\left(|\nabla u|^2-2|\frac{\pa u}{\pa x_2}|^2-k_1^2|u|^2\right)\dd x
    =2\re\Int_{\Om_{1,h}}(x_2+h)g\frac{\pa\ov{u}}{\pa x_2}\dd x,\\ \no
&&\Int_{\Om_{2,h}}\left(|\nabla u|^2-2|\frac{\pa u}{\pa x_2}|^2-\re(k_2^2)|u|^2\right)\dd x\\ \no
&&\qquad\qquad -\int_{\G}\left(\left(|\nabla u^-|^2-\re(k_2^2)|u^-|^2\right)\nu_2
    -2\re\left(\frac{\pa u^-}{\pa\nu}\frac{\pa\ov{u}}{\pa x_2}\right) \right)\dd s\\ \no
&&\qquad\qquad +\Int_{\pa D}(x_2+h)\left(\left(|\nabla u^-|^2-\re(k_2^2)|u^-|^2\right)\nu_2
     - 2\re\left(\frac{\pa u^-}{\pa \nu}\frac{\pa \overline{u}}{\pa x_2}\right)\right)\dd s\\ \label{eq18}
&&\qquad\qquad -\im (k_2^2)\im\Int_{\Om_{2,h}}(x_2+h) u_2\frac{\pa\ov{u}}{\pa x_2}\dd x
 = 2\re\Int_{\Om_{2,h}}g(x_2 + h)\frac{\pa\overline{u}}{\pa x_2}\dd x.
\en
Adding \eqref{eq17} and \eqref{eq18} together gives the Rellich identity
\be\no
&&\left(k^2_1-\re(k^2_2)\right)\Int_{\G}(x_2 + h)|u|^2\nu_2\dd s + 2\Int_{\Om_{h}}|\frac{\pa u}{\pa x_2}|^2\dd x
 = 2h\Int_{\G_h}\left(|\frac{\pa u}{\pa x_2}|^2-|\frac{\pa u}{\pa x_1}|^2 + k_1^2|u|^2\right)\dd s\\ \no
&&+ \Int_{\Om_h}\left(|\nabla u|^2 - \re(k^2) |u|^2\right)\dd x
  + \Int_{\pa D} (x_2 + h) \left(\left(|\nabla u|^2 - \re(k_2^2)|u|^2\right)\nu_2
  - 2\re\left(\frac{\pa u}{\pa \nu}\frac{\pa \overline{u}}{\pa x_2}\right)\right)\dd s\\ \label{eq19}
&&\qquad\qquad -\im(k_2^2)\im\Int_{\Om_{2,h}} (x_2+h) u \frac{\pa \overline{u}}{\pa x_2}\dd x
   - 2\re\Int_{\Om_h}(x_2 + h)g\frac{\pa\ov{u}}{\pa x_2}\dd x.
\en

Now let $D^{\prime}$ be a bounded domain with $\pa D^{\prime}\in C^2$ such that $D\subset D^{\prime}\subset\Om_{2,h}\cup\ov{D}$.
By the global elliptic global regularity estimate we get
\be\label{eq19a}
\|u\|^2_{H^2(D^{\prime}\ba D)}\leq C(\|u\|^2_{L^2(D^{\prime}\ba D)}+\|g\|^2_{L^2(D^{\prime}\ba D)})
\en
Thus, by the trace theorem it follows that
\be\no
&&\Int_{\pa D}(x_2+h)\left(\left(|\nabla u|^2-\re(k_2^2)|u|^2\right)\nu_2
  -2\re\left(\frac{\pa u}{\pa\nu}\frac{\pa\ov{u}}{\pa x_2}\right)\right)\dd s\\ \label{eq20}
&&\quad\le C\|u\|^2_{H^2(D^{\prime}\ba D)}\leq C(\|u\|^2_{L^2(D^{\prime}\ba D)}+\|g\|^2_{L^2(D^{\prime}\ba D)})
\leq C(\|u\|^2_{L^2(\Om_{2,h})}+\|g\|^2_{L^2(\Om_h)}).\qquad
\en
From \eqref{eq8} and \eqref{eq9}, and by using \eqref{eq15}, \eqref{eq16} and the fact that $\beta(x)\geq 0,x\in\G_1$,
we have
\be \label{eq21}
\Int_{\Om_h}\left(|\nabla u|^2-\re(k^2)|u|^2\right)\dd x\leq-\re\Int_{\Om_h}g\overline{u}\dd x,\\ \label{eq22}
\im (k_2^2)\Int_{\Om_{2,h}}|u|^2\dd x+\im\Int_{\G_h}\ov{u}T_1u\dd s\leq\im\Int_{\Om_h}g\ov{u}\dd x.
\en

By \eqref{eq11} and the definition of $T_1$, we have
\be\no
\Int_{\G_h}\left(|\frac{\pa u}{\pa x_2}|^2-|\frac{\pa u}{\pa x_1}|^2 + k_1^2 |u|^2\right)\dd s
\leq 2k_1\im\Int_{\G_h}\ov{u}\frac{\pa u}{\pa x_2}\dd s=2k_1\im\Int_{\G_h}\ov{u} T_1 u\dd s\\ \label{eq22+}
\leq 2k_1\im\Int_{\Omega_h}g\ov{u}\dd x.
\en
Thus, combining \eqref{eq19} and \eqref{eq20}-\eqref{eq22+}, and by the Cauchy-Schwarz inequality, we get
\be\no
&&\left(k^2_1-\re(k^2_2)\right)\Int_{\G}(x_2 + h)|u|^2\nu_2\dd s +\Int_{\Om_{h}}|\frac{\pa u}{\pa x_2}|^2\dd x\\ \label{eq23}
&&\qquad\le C\left(\|g\|_{L^2(\Om_h)}\|u\|_{H^1(\Om_h)}+\|u\|_{L^2(\Om_{2,h})}\|\frac{\pa u}{\pa x_2}\|_{L^2(\Om_{2,h})}
 +\|g\|^2_{L^2(\Om_h)}\right).
\en

Applying Young's inequality to the second term on the right hand side of \eqref{eq23} and using \eqref{eq22}, we obtain
\be\no
&&(k^2_1 - \re (k^2_2))\Int_{\G}(x_2 + h)|u|^2\nu_2\dd s+\Int_{\Om_{h}}|\frac{\pa u}{\pa x_2}|^2\dd x\\ \label{eq24}
&&\qquad\qquad\le C(\|g\|_{L^2(\Om_h)}\|u\|_{H^1(\Om_h)} + \|g\|^2_{L^2(\Om_h)}).
\en
On the other hand, \eqref{eq21} implies that
\be \label{eq3_4}
\|u\|^2_{H^1(\Om_h)}\leq(1+\|k(x)\|_{L^{\infty}(\Om_h)})\|u\|^2_{L^2(\Om_h)}+\|g\|_{L^2(\Om_h)}\|u\|_{L^2(\Om_h)}
\en
Further, we have the following inequality
\be\label{eq25}
\|u\|^2_{L^2(\Om_{1,h})} \leq 2h\|u\|^2_{L^2(\G)} + 2h^2\|\frac{\pa u}{\pa x_2}\|^2_{L^2(\Om_{1,h})}.
\en
which can be proved similarly as in the proof of Lemma 4.3 in \cite{AS}.

Combining \eqref{eq22}, \eqref{eq24}-\eqref{eq25} and the fact that
$\|u\|^2_{L^2(\Om_h)}=\|u\|^2_{L^2(\Om_{1,h})}+\|u\|^2_{L^2(\Om_{2,h})}$, it follows that
\be \label{eq27}
\|u\|^2_{H^1(\Om_h)}\leq C(\|g\|_{L^2(\Om_h)}\|u\|_{H^1(\Om_h)}+\|g\|^2_{L^2(\Om_h)}).
\en
Applying Young's inequality to \eqref{eq27} yields
\ben
\|u\|^2_{H^1(\Om_h)} \leq C\|g\|^2_{L^2(\Om_h)}.
\enn
The proof is complete.
\end{proof}

\begin{theorem}\label{thm2}
For every $\mathcal{G}\in V_h^{\ast}$, the variational problem \eqref{eq_3_1} has a unique solution $u\in V_h$ and
\be \label{eq36}
\| u\|_{V_h} \leq C\|G\|_{V_h^{\ast}}.
\en
In particular, the variational problem \eqref{eq12} or the problem \BVP is well-posed, and the solution satisfies the estimate
\be \label{eq37}
\|u\|_{V_h} \leq C \|g\|_{L^2(\mathbb{R}^2)}.
\en
\end{theorem}

\begin{proof}
From the boundedness of the sesquilinear form $a(\cdot, \cdot)$ and the a priori estimate \eqref{eq13},
and by arguing similarly as in the proof of Lemma 4.5 in \cite{SN3}, we can obtain a the priori estimate \eqref{eq_3_3}.
Then, by Lemma \ref{lem4} the sesquilinear form $a(\cdot, \cdot)$ satisfies the following inf-sup condition
\ben
\inf_{0 \neq u\in V_h} \sup_{0 \neq u\in V_h} \frac{|a(u,v)|}{\|u\|_{V_h} \|v\|_{V_h}} > 0.
\enn
Further, since $a(u,v) = a(\overline{v},\overline{u})$, the following transposed inf-sup condition is also satisfied:
\ben
\sup_{0 \neq u\in V_h} \frac{|a(u,v)|}{\|u\|_{V_h}} > 0 \quad \text{for all} \quad v\in V_h
\enn
Finally, by Theorem \ref{thm1}, we obtain the existence and uniqueness of solution of the variational problem \eqref{eq_3_1}
with the estimate \eqref{eq36}. In particular, the estimate \eqref{eq37} also holds for the problem \BVP, since, in this case,
$\mathcal{G}(v)=-\int_{\Om_{h}}g\ov{v}\dd x$ with the local source $g\in L^2(\R^2)$ compactly supported in $\Om_h$ and
$\|\mathcal{G}\|_{V_h^{\ast}}\leq \|g\|_{L^2(\R^2)}$. The proof is thus complete.
\end{proof}

\subsection{Case 2: $0 <k_2^2 < k^2_1$}

In this subsection, we consider the more challenging case with $k_2^2 > 0$.
In this case, the integrals on $\pa D$ in the Rellich identity \eqref{eq19}
and $\|u\|_{L^2(\Om_{2,h})}$ can not be bounded by \eqref{eq22}, so the a priori
estimate \eqref{eq_3_3} can not be established. However, it is seen from \eqref{eq20} that the integrals
on $\pa D$ in \eqref{eq19} can be bounded locally by $C(\|u\|^2_{L^2(D^{\prime}\backslash D)}
+ \| g \|^2_{L^2(D^{\prime}\backslash D)})$. This fact motivates us to find bounds
for $\| u \|^2_{L^2(D^{\prime}\backslash D)}$ instead of $\|u\|^2_{L^2(\Om_{2,h})}$ with
$D \subset D^{\prime} \subset \Om_{2,h}\cup \overline{D}$. Thus we first consider the variational
problem \eqref{eq_3_1} with the wave number defined by
\ben
k^2_{\alpha}(x) := \left\{
  \begin{array}{ll}
     k^2_1, & x\in \Om_1\\
     k^2_2 + i\alpha, & x\in D^{\prime} \backslash D \\
     k^2_2, & x\in \Om_2 \backslash D^{\prime}
  \end{array}\right.
\enn
with $\alpha > 0$.

\begin{theorem}\label{thm3_5}
The operator equation \eqref{eq_oe} with the wave number $k = k_{\alpha}$ has a unique solution $u \in V_h$, that is, $\mathcal{A}^{-1}_{k_{\alpha}}:V^{\ast}_h\rightarrow V_h$ is bounded.
\end{theorem}

\begin{proof}
By the arguments used in the last subsection, it is sufficient to prove that the a priori estimate in Lemma \ref{lem3_4}
holds with the wave number in the sesquilinear form $a(\cdot, \cdot)$ replaced by $k_{\alpha}$. One can then obtain
the following Rellich identity
\be\no
&&(k^2_1-k^2_2)\Int_{\G}(x_2 + h)|u|^2\nu_2\dd s + 2\Int_{\Om_{h}}|\frac{\pa u}{\pa x_2}|^2 \dd x\\ \no
&&\qquad\qquad=2h\Int_{\G_h}\left(|\frac{\pa u}{\pa x_2}|^2 - |\frac{\pa u}{\pa x_1}|^2 + k_1^2|u|^2\right)\dd s
+\Int_{\Om_h}\left(|\nabla u|^2 - k^2(x) |u|^2\right)\dd x\\ \no
&&\qquad\qquad+\Int_{\pa D}(x_2+h)\left(\left(|\nabla u|^2-\re(k_2^2)|u|^2\right)\nu_2
-2\re\left(\frac{\pa u}{\pa\nu}\frac{\pa\ov{u}}{\pa x_2}\right) \right)\dd s\\ \label{eq3_3}
&&\qquad\qquad- \alpha\,\im\Int_{D^{\prime}\backslash D}(x_2+h) u\frac{\pa\ov{u}}{\pa x_2}\dd x
- 2\re\Int_{\Om_h}(x_2+h)g\frac{\pa\ov{u}}{\pa x_2}\dd x
\en
and that \eqref{eq22} is replaced by
\be \label{eq3_7}
\alpha\Int_{D^{\prime}\backslash D}|u|^2\dd x +\im(\Int_{\G_h}\overline{u}T_1u\dd s)\leq\im(\Int_{\Om_h}g\ov{u}).
\en
It is easy to see that \eqref{eq24} and \eqref{eq3_4} still hold. However, to bound $\|u\|_{L^2(\Om_h)}$,
we first extend $u$ to $\widetilde{u}$ in $\Om_h\cup\ov{D}$ by defining $\widetilde{u}:=u$ in $\Om_h$
and $\widetilde{u}:=v$ in $\ov{D}$, where $v$ is the solution to the Dirichlet problem
$\Delta v=0$ in $D,$ $v=u|_{\pa D}$ on $\pa D.$
Since $u|_{\pa D}\in H^{\frac{1}{2}}(\pa D)$, and by \eqref{eq19a},
$\|v\|^2_{H^1(D)}\leq C\|u\|^2_{H^{1/2}(\pa D)}\leq C\|u\|^2_{H^1(D^{\prime}\ba D)}\leq C(\|u\|^2_{L^2(D^{\prime}\ba D)}
+ \|g\|^2_{L^2(D^{\prime}\ba D)})$. It is clear that $\widetilde{u}\in H^{1}(\Om_h\cup \ov{D})$ and
\be\label{eq3_5}
\|\widetilde{u}\|^2_{L^2(\Om_{h}\cup\ov{D})}\leq 4h\|\widetilde{u}\|^2_{L^2(\G)}
+4h^2\|\frac{\pa\widetilde{u}}{\pa x_2}\|^2_{L^2(\Om_{h}\cup \overline{D})}.
\en
Thus we have
\be \label{eq3_6}
\|u\|^2_{L^2(\Om_h)} \leq C(\|u\|^2_{L^2(\G)} + \|\frac{\pa u}{\pa x_2}\|^2_{L^2(\Om_{h})}
+ \|u\|^2_{L^2(D^{\prime}\ba D)} + \|g\|^2_{L^2(D^{\prime}\backslash D)})
\en
From \eqref{eq20}, \eqref{eq22+}, \eqref{eq3_4}, \eqref{eq3_3}, \eqref{eq3_7} and \eqref{eq3_6}, it follows that
\ben
\|u\|_{V_h} \leq C\|g\|_{L^2(\Om_h)}
\enn
The proof is thus finished.
\end{proof}

We now study the variational formulation with real wave numbers.

\begin{theorem}\label{thm3_1}
For the wave number $k$ satisfying $0<k^2_2<k^2_1$, $\mathcal{A}_k:V_h\rightarrow V_h^{\ast}$ is a Fredholm operator
with index of zero.
\end{theorem}

\begin{proof}
Define the restriction operator $\mathcal{P}:V_h\rightarrow V_h^{\ast}$ such that $\mathcal{P}u=u|_{D^{\prime}\ba D}$
for $u\in V_h$. Then $\mathcal{P}$ is compact. This can be seen by the facts that the embedding
$V_h\rightarrow H^1(D^{\prime}\ba D)$ is bounded, the embedding $H^1(D^{\prime}\ba D)\rightarrow L^2(D^{\prime}\ba D)$
is compact and the embedding $L^2(D^{\prime}\ba D)\rightarrow V_h^{\ast}$ is bounded.
Then, by the definition of $\mathcal{A}_k$ and $k_{\alpha}$, we have
$\mathcal{A}_{k} = \mathcal{A}_{k_{\alpha}} - i\alpha \mathcal{P}$. Thus $\mathcal{A}_{k}u=\mathcal{G}$ can be rewritten
as $(\mathcal{A}_{k_{\alpha}} - i\alpha\mathcal{P})u = \mathcal{G}$, where $\mathcal{A}_{k_{\alpha}}$ is an isomorphism
and $\mathcal{P}$ is compact from $V_h$ to $V_h^{\ast}$. Hence, it follows that $\mathcal{A}_{k}$ is a Fredholm operator
of index zero.
\end{proof}

\begin{corollary}\label{coro3_1}
Let the wave number $k$ satisfy the condition in Theorem \ref{thm3_1}. If $\text{m}(\G_1)\neq 0,\beta >0$ on $\G_1$,
where $\text{m}(\G_1)$ denotes the measure of $\G_1$ on the boundary $\pa D$, then there exists a unique
solution to \eqref{eq_oe}. In particular, the variational problem \eqref{eq12} or the problem \BVP is well-posed
with the solution satisfying the estimate \eqref{eq37}.
\end{corollary}

\begin{proof}
From Theorem \ref{thm3_1}, the existence follows from the uniqueness. It is sufficient to prove that if $u\in V_h$
satisfying \eqref{eq_oe} with $\mathcal{G} = 0$ then $u$ vanishes in $\Omega_h$. Let $v = u$ in \eqref{eq_3_1} and
take the real part of the equation \eqref{eq_3_1}. One obtains that
\ben
\Int_{\G_1}\beta |u| \dd s = 0
\enn
Since $\beta >0$ on $\G_1$, we have $u = 0$ on $\Gamma_1$, which, together with the boundary condition \eqref{eq3},
implies that $\frac{\pa u}{\pa \nu} = 0$ on $\G_1$. By Holmgren's uniqueness theorem, $u$ vanishes in $\Omega_h$.
The well-posedness of the variational problem \eqref{eq_oe} or the problem \BVP follows by the same argument as
used in Theorem \ref{thm2}.
\end{proof}

\begin{remark}\label{re1} {\rm
In Corollary \ref{coro3_1}, the direct scattering problem is well-posed if the buried obstacle $D$ is partially
coated with a non-absorbing material. However, similar results can not be generalized to the cases with other
boundary conditions (e.g., the Dirichlet or Neumann boundary condition or the mixed Dirichlet and Neumann condition)
on the obstacle since the uniqueness of solutions is not clear in these cases.
}
\end{remark}

In the end of this section, we give the following corollary which will be used in the proof of Theorem \ref{thm5_1}.

\begin{corollary}\label{coro4_1}
Assume that the wave numbers satisfy the condition in Case 1 or Case 2 and that part of the boundary of the buried obstacle
is dielectric. Let $f\in H^{-1}(\R^2)$ and let $\chi$ be a cut off function with a compact support $K\subset \Om_h$.
Then there exists exactly one solution $u\in V_h$ to \BVP with $g$ replaced by $\chi f$. Further, the solution $u$ satisfies
the estimate
\ben
\|u\|_{V_h} \leq C\|f\|_{H^{-1}(\R^2)}
\enn
where the constant $C > 0$ is independent of $f$.
\end{corollary}

\begin{proof}
We first claim that $\chi f\in (H^{1}(K))^{\ast}$. In fact, for $\varphi \in C^{\infty}(\R^2)\cap H^1(\R^2)$
\ben
|\langle \chi f,\varphi \rangle | = | \langle f, \chi\varphi\rangle | \leq C
\|f\|_{H^{-1}(\R^2)}\|\varphi\|_{H^1(K)}
\enn
where $\langle \cdot,\cdot \rangle$ is the dual pair between $H^{-1}(\R^2)$ and $H^1(\R^2)$ and $C$ is independent of $f$.
By the argument of density, we obtain $\|\chi f\|_{(H^{1}(K))^{\ast}}  \leq C \|f\|_{H^{-1}(\R^2)}$.
Meanwhile, noting that $K \subset \Om_h$, it is clear that $V_h \subset H^1(\Om_h) \subset H^1(K)$.
Therefore $\chi f \in (H^{1}(K))^{\ast} \subset V_h^{\ast}$ and $\|\chi f\|_{V_h^{\ast}}\leq C\|\chi f\|_{(H^{1}(K))^{\ast}}$.
Thus one has the following variational problem
\ben
a(u,v) = \langle \chi f,v\rangle_{V_h}
\enn
which, by Theorem \ref{thm2} and Corollary \ref{coro3_1}, is well-posed. The proof is thus complete.
\end{proof}

\begin{remark}\label{rem3}{\rm
All the results in this section also hold under the condition that $k^2_1 < \re{k^2_2}$. In fact,
Lemma \ref{lem3_4} holds on noticing that applying Green's first theorem to $u$ and
$\varphi_{A}(r)(x_2-h)\pa\overline{u}/\pa x_2$ leads to a Rellich identity similar to \eqref{eq3_3}.
Then other results follow naturally after repeating the argument in this section again.
And they also have the natural generalization in the cases of higher dimensions.
}
\end{remark}

\begin{remark}\label{rem3+}{\rm
It is known from the proof of Lemma \ref{lem3_4} and Theorem \ref{thm3_5} that, for $D=\varnothing$,
the direct scattering problem is well-posed if $\im(k_2^2)\geq 0$ and either $0<k_1^2<\re(k_2^2)$ or $k_1^2>\re(k_2^2)>0$.
}
\end{remark}

In the remaining part of this paper, we always assume that one of the following conditions is satisfied,
under which the direct scattering problem is well-posed:

(i) $0<k^2_1<\re{k^2_2}$ or $k^2_1>\re{k^2_2}>0$, $\im{k^2_2}>0$, and any boundary condition on the obstacle.

(ii) $0<k^2_1<k^2_2$ or $k^2_1>k^2_2>0$ and part of the obstacle is partly coated.

\section{The scattering problem with incident PSWs and HSPSWs}\label{sec4}
\setcounter{equation}{0}

In this section, we study the well-posedness of the scattering problem corresponding to incident
point source waves (PSWs) and hyper-singular point source waves (HSPSWs).
The first case corresponds to the problem \BVP with $g$ being a Dirac delta function,
saying $\delta(x-z), z\in\mathbb{R}^2\backslash\{\overline{D}\cup \G\}$,
while for the second case, $g=\delta_1^{\prime}(x-z)$ where $\delta_1^{\prime}(x)$ stands for the derivative
of $\delta(x)$ with respect to $x_1$ in the distributional sense. Obviously, Theorem \ref{thm2} can not be
applied directly since the distributions $\delta(x-z)$ and $\delta_1^{\prime}(x-z)$ do not belong into $V_h^{\ast}$.
However, we shall see shortly that both cases can be modified into the case that one can deal with by Theorem \ref{thm2}
and Corollary \ref{coro3_1} and we only consider the case when the point source lies upon the rough interface.

Let $\Phi_k(x;z):=(i/4)H_0^{(1)}(k|x-z|),\; x,z\in\mathbb{R}^2,\;x\neq z$ denote the
fundamental solution of the Hemholtz operator $\Delta + k^2$ with $H_0^{(1)}$ the Hankel
function of the first kind of order zero. For $z=(z_1,z_2)\in U_0^+$,
define $z^{\prime}=(z_1,-z_2)$. Then $G_k(x;z)=\Phi_k(x;z)-\Phi_k(x;z^{\prime})$
is the Dirichlet Green's function for the Hemholtz operator $\Delta + k^2$ in $U_0^+$.
By the asymptotic property of the Hankel function for small and large arguments,
$G_k$ satisfies the following inequalities:
\begin{equation}\label{eq4_1}
\begin{split}
&|G_k(x;z)|,|\nabla_x G_k(x;z)|, |\nabla_z G_k(x;z)|\leq C\frac{(1+|x_2|)(1+|z_2|)}{|x-z|^{3/2}}
\quad \text{for} \quad x,z\in U^+_0\text{with}\;|x-z|\geq 1, \\
&|G_k(x;z)|\leq C(1+|\text{log}|x-z||)\quad \text{for} \quad x,z\in U^+_0 \text{ with}\;0<|x-z|\leq 1,\\
&  |\nabla_x G_k(x;z)|, |\nabla_z G_k(x;z)| \leq \frac{C}{|x - z|} \quad \text{for}
\quad x,z\in U^+_0 \text{ with }  0 < |x-z| \leq 1,
\end{split}
\end{equation}
where $C$ is a positive constant depending only on $k$.

It is easy to verify that $(\Delta+ k^2)(\pa\Phi_{k_1}(x)/\pa x_1)=\delta_1^{\prime}(x)$ in
the distributional sense. Therefore, $\pa\Phi_k(x)/\pa x_1, x\neq 0$, is the HSPSW positioned at the origin.
Since $G_{k_1}(\cdot;z)$ and $ G_{k_1}^{\prime}(\cdot;z):= \pa G_{k_1}(\cdot;z)/ \pa x_1$ both belong
to $L^1_{\text{loc}}(\R^2)$ and $H^1(\Om_h\backslash B_{\varepsilon}(z)), \varepsilon > 0$, for convenience,
we may use $G_{k_1}(x;z)$ (or $G_{k_1}^{\prime}(x;z)$) to denote the incident PSW (or incident HSPSW).
Consider the incident field $u^{\text{i}}\in\{ G_{k_1}(\cdot;z),G_{k_1}^{\prime}(\cdot;z)\}$.
We write the total field $u^{\text{t}} = u^{\text{i}} + u$ in $\Om_1$ and $u^{\text{t}} = u$ in $\Omega_2$
with $u$ being the transmitted field.
Note that the total field corresponding to PSW and HSPSW does not belong to $V_h$ because of the singularity
of the incident field. However, for a source positioned at $z\in \Om_1$ and $0 < \delta_0 < \text{dist}(z,\G)$,
it is expected to find the solution in the space
$\widetilde{V_h}:= \{u| u\in H^1(\Om_h\backslash B_{\delta_0}(z)), u|_{\G_2} = 0\}$
with the norm $\|v\|_{\widetilde{V_h}}:=\|v\|_{H^1(\Om_h\backslash B_{\delta_0}(z))}$.

{\bf The scattering problem (SP):} For $z\in\Om_1$ and $u^{\text{i}}\in\{ G_{k_1}(\cdot;z),G_{k_1}^{\prime}(\cdot;z)\}$,
find $u^{\text{t}}(x;z)\in \widetilde{V_h}$, such that
\ben
u^{\text{t}}(\cdot;z) = u(\cdot;z) + u^{\text{i}}(\cdot;z)\quad\text{in}\quad \Om_1 \\
u^{\text{t}}(\cdot;z) = u(\cdot;z) \quad\text{in}\quad\Om_2\backslash D,
\enn
where $u(\cdot;z)$ satisfies
\ben
\Delta u(\cdot;z) + k^2 u(\cdot;z) = 0\quad\text{in}\quad\R^2\backslash(\overline{D}\cup\G)\\
 u^+(\cdot;z)- u^-(\cdot;z)= -u^{\text{i}}\;\;\text{on}\;\;\G,\\
\frac{\pa u^+}{\pa\nu}(\cdot;z)-\frac{\pa u^-}{\pa\nu}(\cdot;z)=-\frac{\pa u^{\text{i}}}{\pa\nu}\;\;\text{on}\;\;\G,
\enn
and the boundary condition \eqref{eq3} and the radiation conditions \eqref{eq4} and \eqref{eq5}.

We now study the existence and uniqueness of solutions to the scattering problem \SP by replacing the incident wave
with a non-singular one. This technique has been used in \cite{SN6}. Choose $\delta>0$ such that
$\text{dist}(z,\G)>\delta$ and define a new incident wave by
\ben
\wi{u}^{\text{i}}(x)=\left\{
  \begin{array}{ll}
  u^{\text{i}}(x), & x\not \in B_{\delta}(z)\\
  A+B J_0(x), & \text{otherwise}
  \end{array}
  \right.
\enn
where the constants $A$ and $B$ are chosen to ensure that $\wi{u}^{\text{i}}\in C^1(\Om_{1,h})$.
Then $\wi{u}^{\text{i}}\in H^2_{\text{loc}}(\Om_{1,h})$ and $(\Delta+k_1^2)\wi{u}^{\text{i}}=\wi{g}$,
where $\wi{g}(x):=Ak_1^2\;\text{for}\;x\in B_{\delta}(y),\wi{g}(x):=0$ otherwise.
Since $u^{\text{i}}=\wi{u}^{\text{i}}$ outside $B_{\delta}(z)$ and $B_{\delta}(z)\cap\G=\varnothing$, $u^{\text{i}}$
and $\wi{u}^{\text{i}}$ have the same boundary value and normal derivative on $\G$. Obviously, the substitution of
$u^{\text{i}}$ by $\wi{u}^{\text{i}}$  does not change the scattered field, so we can reformulate the scattering
problem \SP by finding $\widetilde{u}^{t}(x)=\wi{u}^{\text{i}}(x)+u(x)$ for
$x\in\ov{\Om}_{1,h},\widetilde{u}^{t}=u(x)$ for $x\in\overline{\Om}_{2,h}$ such that $\widetilde{u}^t$
solves \BVP with $g:=\wi{g}\in L^2(\R^2)$. By Theorem \ref{thm2} and Corollary \ref{coro3_1},
there is a unique solution $\widetilde{u}^t\in V_h$ to the scattering problem \SP. Then we obtain the scattered field
$u(x) = \widetilde{u}^t(x)-\wi{u}^{\text{i}}(x)$ for $x\in\overline{\Om}_{1,h},$ $u(x)=\widetilde{u}^t(x)$ for
$x\in\overline{\Om}_{2,h}$. It is clear that $u|_{\Om_{j,h}}\in H^1(\Om_{j,h})$. Since $G_{k_1}(x;z)$ and $G_{k_1}(x;z)$
satisfy the inequality \eqref{eq4_1}, we have $u^{\text{t}}(\cdot;z)\in\widetilde{V_h}$.

For $z\in\Om_1$ and $\delta>0$, the solution to the scattering problem \SP satisfies that
$u^{\text{t}}(\cdot;z)\in\widetilde{V_h}$. Let $K$ be a compact set in $\Om_h\backslash B_{\delta}(z)$.
Then it is clear that $\|u^{\text{t}}(\cdot;z)\|_{H^1(K)}$ is bounded. In fact, it can be shown that,
as $z$ approaches $\G$, $\|u^{\text{t}}(\cdot;z) \|_{H^1(K)}$ is bounded uniformly. This is shown in the following theorem,
which will be one of the key ingredients in proving the uniqueness for the inverse scattering problem.

\begin{theorem}\label{thm4_1}
For $z_0\in\R^2\ba D$ fixed, $\delta>0$, $z\in B_{\delta}(z_0)\ba\G$ with $\ov{B_{\delta}(z_0)}\cap\ov{D}=\varnothing$
and the compact set $K\subset\Om_h\backslash B_{\delta}(z_0)$, assume that $d=\text{dist}(K,\ov{B_{\delta}(z_0)})>0$.
Then the total field $u^{\text{t}}(\cdot;z)$ of the scattering problem \SP satisfies the estimate
\be\label{eq4_3}
\|u^{\text{t}}(\cdot;z)\|_{H^1(K)}\leq C,
\en
where the constant $C > 0$ depends on $k_1,k_2,\delta,d$ but is independent of $z$.
\end{theorem}

\begin{proof}
We only consider the case that $z_0\in\overline{\Om_1}$, $z\in\Om_1$ and the incident field $u^{\text{i}}$ is a HSPSW.
The other cases can be proved similarly. For the case that the incident field is a PSW, see Remark \ref{rem4_1}.

Take a smooth cut-off function $\chi(x)$ such that $\chi(x)=1$ for $x\in B_{\delta}(z_0),\chi(x)=0$ for
$x\in B^c_{\delta + d/2}(z_0)$, and $0<\chi(x)<1$. Then the total field can be written as
$u^{\text{t}}(x)=\chi G^{\prime}_{k_1}(x;z)+V(x;z)$. It is clear that $V(\cdot;z)$ satisfies
\ben
(\Delta + k^2) V(x;z) = \widetilde{g}_z(x)
\enn
where $\widetilde{g}$ is defined by
\ben
\widetilde{g}_z(x)=\left\{\begin{array}{ll}
-\Delta\chi G^{\prime}_{k_1}(x;z)-2\nabla\chi\cdot\nabla G^{\prime}_{k_1}(x;z) &\text{in}\;\;\Om_1,\\
-\Delta\chi G^{\prime}_{k_1}(x;z)-2\nabla\chi\cdot\nabla G^{\prime}_{k_1}(x;z)
+ \chi(k^2_1 - k^2_2)G^{\prime}_{k_1}(x;z) &\text{in}\;\; \Om_2,
\end{array}\right.
\enn
and by the definition of $\chi$, $V(\cdot;z)$ satisfies the transmission condition \eqref{eq2},
the boundary condition \eqref{eq3} and the radiation conditions \eqref{eq4}-\eqref{eq5}.

We claim that $\widetilde{g}_z\in H^{-1}(\R^2)$ with a compact support. Since
$-\Delta\chi G^{\prime}_{k_1}(\cdot;z)-2\nabla\chi \cdot\nabla G^{\prime}_{k_1}(\cdot;z)\in L^2(\R^2)$ supported
in $B_{\delta + d/2}(z_0)$, we only need to prove that $\widetilde{f}_z\in H^{-1}(\R^2)$ with
$\widetilde{f}_z:=0\;\;\text{in}\;\;\Om_1,\widetilde{f}_z:=\chi(\cdot)(k^2_1-k^2_2)G^{\prime}_{k_1}(\cdot;z)\;\;\text{in}\;\;\Om_2$.
In fact, it is seen from \eqref{eq4_1} that $G^{\prime}_{k_1}(\cdot;z)\in L^{p}_{loc}(\R^2)$ with $1\leq p<2$.
Then, taking $1<p<2$, we have $\widetilde{f}_z\in L^p(\R^2)$ with a compact support. By the standard Sobolev embedding
theorem, we have that $H^1_0(\R^2)\hookrightarrow L^q$ for all $2\leq q<\infty$.
Thus $\widetilde{f}_z\in L^p(\R^2)\subset H^{-1}(\R^2)$ for $1<p<2$. Furthermore, we have
\ben
\|\widetilde{g}\|_{H^{-1}(\R^2)}\leq\|\chi\|_{H^2(B_{\delta+d/2}(z_0)
\ba B_{\delta}(z_0))}\|G^{\prime}_{k_1}\|_{H^1( B_{\delta + d/2}(z_0)\ba B_{\delta}(z_0))}\\
+|k^2_1-k^2_2|\|\chi\|_{H^1(B_{\delta+d/2}(z_0))}\|G_{k_1}\|_{L^2(B_{\delta+d/2}(z_0))}\leq C(k_1,k_2,\delta,d)
\enn
Since $\widetilde{g}_z$ is compactly supported in $B_{\delta + d/2}(z_0)$, then for another smooth cut-off function
$\widetilde{\chi}$ such that $\widetilde{\chi} = 1$ in $B_{\delta + d/2}(z_0)$ and $\widetilde{\chi}=0$ outside
$B_{\delta+ 3d/4}(z_0)$, we have $\widetilde{\chi}\widetilde{g}_z = \widetilde{g}_z$.
From Corollary \ref{coro4_1}, we see that $\|V(x;z)\|_{\widetilde{V_h}}\leq C(\delta,d)\|\widetilde{g}\|_{H^{-1}(\R^2)}
\leq C(k_1,k_2,\delta,d)$ for any $z\in B_{\delta}(z_0)\cap\Om_1$.
Since $u^{\text{t}}(x)=\chi G^{\prime}_{k_1}(x;z)+V(x;z)$, and by the definition of $\chi$, \eqref{eq4_3} holds.
\end{proof}

\begin{remark}\label{rem4_1}{\rm
(i) For \SP due to PSW, by the same argument used in Theorem \ref{thm4_1}, we can conclude that
$u^{\text{t}}(\cdot;z)\in H^2(K)$ and $\|u^{\text{t}}(\cdot;z)\|_{H^2(K)}$ is bounded uniformly with respect to $z$.

(ii) In the case that $z \in B_{\delta}(z_0)$ with $\text{dist}(B_{\delta}(z_0),\G)=d>0$, $u(\cdot;z)\in H_{loc}^1(\Om_j),$
$j=1,2$, and $\|u(\cdot;z)\|_{H_{loc}^1(\Om_j)}$ is bounded uniformly with respect to $z$.
}
\end{remark}

Denote by $u(\cdot;z)$ ($u^{\text{t}}(\cdot;z)$) the scattered (total) field corresponding to an incident PSW 
with the source position $z$ and by $u^{\prime}(\cdot;z)$ ($u^{\prime \text{t}}(\cdot;z)$) the scattered (total) 
field corresponding to an incident HSPSW positioned at $z$. 
Define $\widehat{V_h}:=\{u\;|\;u|_{\Om_{1,h}}\in H^1(\Om_{1,h}),u|_{\Om_{2,h}}\in H^1(\Om_{2,h}),u|_{\G_2}=0\}$ 
with the norm $\|u\|_{\widehat{V_h}}=\|u\|_{H^1(\Om_{1,h})}+\|u\|_{H^1(\Om_{2,h})}$. The following Theorem gives 
an important relation between $u(\cdot;z)$ and $u^{\prime}(\cdot;z)$.

\begin{theorem}\label{thm3}
For $z\in\Om_1$, the limit
\ben
\frac{\pa u(\cdot;z)}{\pa z_1}:=\lim_{\varepsilon\rightarrow 0}\frac{u(\cdot;z+\varepsilon\mathbf{e}_1)-u(\cdot;z)}{\varepsilon}
\enn
exists in $\widehat{V_h}$, where $\mathbf{e}_1=(1,0)^T$. Further, $u^{\prime}(\cdot;z)=-\frac{\pa u(\cdot;z)}{\pa z_1}$
\end{theorem}

\begin{proof}
It is sufficient to show that
\ben
\lim_{\varepsilon\rightarrow 0}v_{\varepsilon}:=\lim_{\varepsilon\rightarrow 0}\left(\frac{u(\cdot;z+\varepsilon\mathbf{e}_1)
- u(\cdot;z)}{\varepsilon} + u^{\prime}(\cdot;z)\right) = 0
\enn
in $\widehat{V_h}$.

Noting that $\pa G_{k_1,z}(x;z)/\pa x_1 = - \pa G_{k_1,z}(x;z)/\pa z_1,x\in \G$, it is clear that $v_{\varepsilon}$ is
the solution of the scattering problem \SP with $u^{\text{i}}=u_{\varepsilon}^{\text{i}}:=(G_{k_1,z}(x;z+\varepsilon\mathbf{e}_1)
- G_{k_1,z}(x;z))/\varepsilon-\pa G_{k_1,z}(x;z)/\pa z_1$. Then the total field $u_{\varepsilon}^{\text{t}}=v_{\varepsilon}
+ u_{\varepsilon}^{\text{i}}\;\; \text{in}\;\;\Om_1,\;u_{\varepsilon}^{\text{t}}=v_{\varepsilon}\;\;\text{in}\;\;\Om_2$.
Since $z\in\Om_1$, there exists a $r>0$ such that $\overline{B_{r}(z)}\in\Om_1$. We may assume that $|\varepsilon|<r/8$.
To study the asymptotic property of $v_{\varepsilon}$ as $\varepsilon\rightarrow 0$, we first take a smooth cut-off
function $\chi(x)$, such that $\chi(x)=1$ for $x\in B_{r/4}(z),$ $\chi(x)=0$ for $x\in B^c_{r/2}(z)$.
The total field can be written as $u_{\varepsilon}^{\text{t}}=\chi u_{\varepsilon}^{\text{i}}+V_{\varepsilon}\;\;\text{in}\;\;\Om_1,$ $u_{\varepsilon}^{\text{t}}=V_{\varepsilon}\;\;\text{in}\;\;\Om_2$. Then $V_{\varepsilon}$ satisfies that
\ben
\Delta V_{\varepsilon}+k^2 V_{\varepsilon}=g_{\varepsilon}\quad\text{in}\;\;\R^2,
\enn
where $g_{\varepsilon}=-(\Delta\chi u_{\varepsilon}^{\text{i}}+2\nabla\chi\cdot\nabla u_{\varepsilon}^{\text{i}})$.
Moreover, $V_{\varepsilon}$ satisfies the transmission condition \eqref{eq2}, the boundary condition \eqref{eq3} and
the radiation conditions \eqref{eq4} and \eqref{eq5}. By the definition of $\chi$, we see that $g_{\varepsilon}$ is a
smooth function compactly supported in $B_{r/2}(z)\backslash B_{r/4}(z)$ and satisfies the estimate
\ben
\|g_{\varepsilon}\|_{L^2(\R^2)}=\|g\|_{L^2(B_{r/2}(z)\backslash B_{r/4}(z))}\leq C \|u_{\varepsilon}^{\text{i}}\|_{H^1(B_{r/2}(z)\backslash B_{r/4}(z))}
\enn
where the constant $C$ is independent of $\varepsilon$.
From Theorem \ref{thm2}, we have $\|V_{\varepsilon}\|_{V_h}\leq C\|g_{\varepsilon}\|_{L^2(\R^2)}$.
Then $v_{\varepsilon}=(1-\chi)u_{\varepsilon}^{\text{in}}+V_{\varepsilon}\;\;\text{in}\;\;\Om_1,$
$v_{\varepsilon}=V_{\varepsilon}\;\;\text{in}\;\;\Om_2$. Thus,
\begin{equation}\label{v_ep}
\begin{aligned}
\|v_{\varepsilon}\|_{H^1(\Om_{1,h})}+\|v_{\varepsilon}\|_{H^1(\Om_{2,h})}
\leq C(\|u_{\varepsilon}^{\text{i}}\|_{H^1(\Om_{1,h}\backslash B_{r/2}(z))}+\|V_{\varepsilon}\|_{V_h})\\
\leq C(\|u_{\varepsilon}^{\text{i}}\|_{H^1(\Om_{1,h}\backslash B_{r/2}(z))}+\|g_{\varepsilon}\|_{L^2(\R^2)})
\leq C\|u_{\varepsilon}^{\text{i}}\|_{H^1(\Om_{1,h}\backslash B_{r/4}(z))}
\end{aligned}
\end{equation}

By the asymptotic property of the point source and its derivatives at infinity we have
$u_{\varepsilon}^{\text{i}} \in H^1(\Om_{1,h}\backslash B_{r/4}(z))$. Then for a given $\eta >0$, there exists a $M > 0$
such that $\|u_{\varepsilon}^{\text{i}}\|_{H^1(\R^2\backslash\Om_{1,h}(M))}\leq\eta/2 $ for all $\varepsilon$ satisfying
$|\varepsilon| < r/8$. Since $|\varepsilon|< r/8 $, by the interior elliptic regularity and the fact that
$u_{\varepsilon}^{\text{i}} \in H^1(\Om_{1,h}\backslash B_{r/4}(z))$, there exists a $\delta > 0$ such that  $\|u_{\varepsilon}^{\text{i}}\|_{H^1(\Om_{1,h}(M)\backslash B_{r/4}(z))} < \eta/2$ for $|\varepsilon|<\delta$.
Therefore, for any $\forall\eta >0$ there exists a $\delta >0$ such that
$\|u_{\varepsilon}^{\text{i}}\|_{H^1(\Om_{1,h}\backslash B_{r/4}(z))}<\eta$ for $|\varepsilon|<\delta$
so, by \eqref{v_ep} $\|v_{\varepsilon}\|_{H^1(\Om_{1,h})} + \|v_{\varepsilon}\|_{H^1(\Om_{2,h})}\leq C\eta$.
This means that $v_{\varepsilon}\rightarrow 0$ in $\widehat{V_h}$ as $\varepsilon \rightarrow 0$.
The proof is thus finished.
\end{proof}

\begin{remark}{\rm
Theorem \ref{thm3} also holds in higher dimensions. Nevertheless, Theorem \ref{thm4_1}
can only be proved in two and three dimensions up to now. In fact, for the case that $n \geq 3$ in
Theorem \ref{thm4_1}, $\widetilde{g}_z\in L^p(\R^n)$ with a compact support if $1\leq p<n/(n-1)$
and $H_0^1(\R^n)\hookrightarrow L^{2n/(n-2)}(\R^n)$ or equivalently,
$L^{2n/(n+2)}(\R^n)\hookrightarrow H^{-1}(\R^2)$. Then for a function
$f$ in $L^p(\R^n)$ with $1\leq p\leq\infty $ with a compact support,
by the Cauchy-Schwarz inequality we have that $f\in L^q(\R^n)$ if $1\leq q\leq p$.
Thus we can conclude that $\widetilde{g}_z\in H^{-1}(\R^n)$ if there exists a $p$ such
that $2n/(n+2)<p<n/(n-1)$ which implies that $n < 4$. By Corollary \ref{coro4_1},
Theorem \ref{thm4_1} holds in two and three dimensions. However, it is not clear
whether the same conclusion in dimensions $n \geq 4$ holds since
the solvability of the boundary value problem \BVP with the right hand $g\in L^p(1\leq p<n/(n-1))$ is unknown yet.
}
\end{remark}

\section{The inverse scattering problem}\label{sec5}
\setcounter{equation}{0}

In this section, we consider the inverse problem of recovering the interface and the buried obstacle with its physical 
property simultaneously from the scattered field generated by PSWs. Suppose the scattered fields are generated by PSWs 
with source positions located on the line segment $\Sigma_s\subset\G_b$ and measured on another line segment 
$\Sigma_r\subset\G_c$; see Figure \ref{fig1}. Then the inverse scattering problem can be stated as follows.

{\bf Inverse scattering problem(ISP):} Given the wave numbers $k_j,j=1,2$, and the scattered field $u(x;z)$ for 
$z\in\Sigma_s\subset\G_b, x\in\Sigma_r\subset\G_c$, determine the rough interface $\G$, the obstacle $D$ and its 
physical property $\mathcal{B}$.

In the proof of the uniqueness result for the inverse scattering problem, the recovery of the rough interface will be 
achieved by constructing a special transmission problem called \textbf{Interior Transmission Problem (ITP)}. 
In recent years, there have been a great development on the study of interior transmission problem and the associated 
transmission eigenvalues (see, e.g. \cite{FC2,CCC08,CPS07,S11}). 
Recently, in \cite{YZZ} the interior transmission problem is exploited to study the inverse problem for the bounded 
penetrable obstacle scattering problem. We will use the same idea to prove the uniqueness of the inverse rough surface 
scattering problem. To this end, we first briefly recall some background on the interior transmission problem. 
We define space
\ben
H^1_{\Delta}(\Om):=\{w\in H^1(\Om)\;|\;\Delta w\in L^2(\Om)\}
\enn
equipped with the norm $\|w\|^2_{H^1_{\Delta}(\Om)}=\|w\|^2_{H^1(\Om)}+\|\Delta w\|^2_{L^2(\Om)}$. 
It is clear that $ H^1_{\Delta}(\Om)$ is a Hilbert space. Moreover, a function $w\in H^1_{\Delta}(\Omega)$ has 
traces $\gamma_0 w\in H^{\frac{1}{2}}(\pa\Om)$ and $\gamma_1 w\in H^{-\frac{1}{2}}(\pa\Om)$. 
In particular, we set $H^2_0(\Omega):=\{w\in H^1_\Delta(\Omega)\;|\;\gamma_0 w=\gamma_1 w=0\}$. 
Let $n(x)$ be the index of refraction such that $\im(n)\geq 0$. Suppose that either $1+r_0<\re{(n)}<\infty$
or $0<\re{(n)}<1-r_1$, where $r_0,r_1>0$.

{\bf Interior Transmission Problem (ITP):}
Given $(f_1,f_2)\in\{(\gamma_0 w,\gamma_1 w)\;|\;w\in H^1_{\Delta}(\Om)\}$, find $U,V\in L^2(\Om)$ 
such that $U-V-w\in H^2_0(\Omega)$ and satisfying that
\ben
\Delta V+ k^2V=0\quad\text{in}\;\;\Om,\\
\Delta U+k^2nU=0\quad\text{in}\;\;\Om,\\
U-V=f_1,\;\;\frac{\pa U}{\pa\nu}-\frac{\pa V}{\pa\nu}=f_2\quad\text{on}\;\;\pa\Om
\enn

We say $k^2$ is an interior transmission eigenvalue of \ITP if the homogenous problem has a nonzero solution. 
An interior transmission problem is well-possed if $k^2$ is not an interior transmission eigenvalue. 
In the following theorem, we collect some results about the wellpossedness of \ITP and properties of the 
interior transmission eigenvalues.
    
\begin{theorem}\label{thm4}
\item[(i)] Let $\Omega$ be fixed. If $\im(n)>0$, then \ITP is well-posed. If $\im(n)=0$, then there 
exists an infinite set of transmission eigenvalues of \ITP with the only accumulation point at the infinity. 
Moreover, if $k^2$ is not an eigenvalue, then \ITP has a unique solution $(U,V)\in L^2(\Omega)\times L^2(\Omega)$ 
such that
\be\label{eq:stability}
\|U\|_{L^2(\Om)}+\|V\|_{L^2(\Om)}\leq C(\|f_1\|_{H^{\frac{1}{2}}(\pa\Om)}+\|f_2\|_{H^{-\frac{1}{2}}(\pa\Om)})
\en

\item[(ii)] Let $k>0$ be fixed and assume that $\im(n)=0$. If the diameter of the domain $\Om$ is small enough, 
then $k^2$ can not be an interior transmission eigenvalue of \ITP. Moreover, in such case, the 
estimate \eqref{eq:stability} holds.
\end{theorem}

The proof of Theorem \ref{thm4} (i) can be found in \cite{FC2}. Theorem \ref{thm4} (ii) was proved in \cite{YZZ}. 
The following reciprocity relation about the total field (as well as the scattered field) induced by the point sources 
will also be useful.

\begin{theorem}{\bf (Reciprocity relation)}
For $z_1,z_2 \in \R^2\backslash \{\G \cup \overline{D}\}$ and $z_1 \neq z_2$, the total field satisfies
\ben
u^{\text{t}}(z_1;z_2) = u^{\text{t}}(z_2;z_1)
\enn
\end{theorem}

\begin{proof}
We only consider the case $z_1 \in \Om_1, z_2\in \Om_2\backslash \overline{D}$, the other cases can be treated similarly.

For $A > 0$ and $h > \text{max}(|z_1|,|z_2|)$, take $\varepsilon > 0$ such that $B_{\varepsilon}(z_1) \subset \Om_{1,h}(A)$
and $B_{\varepsilon}(z_2)\subset\Om_{2,h}(A)$. Since $u^{\text{t}}(\cdot;z_j)\in H^2_{\text{loc}}(\Om_{j,h}(A)\backslash B_{\varepsilon}(z_j))$,
we apply Green's first theorem to $u^{\text{t}}(x;z_1)$ and $u^{\text{t}}(x;z_2)$ in
$\Om_{1,h}(A)\backslash B_{\varepsilon}(z_1)$ to get
\begin{equation}\label{eq55}
\begin{aligned}
0 = \Int_{\Om_{1,h}(A)}\left(u^{\text{t}}(x;z_1)\Delta u^{\text{t}}(x;z_2)
- u^{\text{t}}(x;z_2)\Delta u^{\text{t}}(x;z_1)\right) \dd x = \\
\Int_{\G_{h}(A)}\left(u^{\text{t}}(x;z_1)\frac{\pa u^{\text{t}}}{\pa \nu}(x;z_2)
- u^{\text{t}}(x;z_2)\frac{\pa u^{\text{t}}}{\pa \nu}(x;z_1)\right) \dd s + R_1(A) - \\
\Int_{\G(A)}\left(u^{\text{t}}(x;z_1)|^+\frac{\pa u^{\text{t}}}{\pa \nu}(x;z_2)|^+
- u^{\text{t}}(x;z_2)|^+\frac{\pa u^{\text{t}}}{\pa \nu}(x;z_1)|^+\right)\dd s - \\
\Int_{\pa B_{\varepsilon}(z_1)} \left( u^{\text{t}}(x;z_1)\frac{\pa u^{\text{t}}}{\pa \nu}(x;z_2)
- u^{\text{t}}(x;z_2)\frac{\pa u^{\text{t}}}{\pa \nu}(x;z_1)\right)\dd s
\end{aligned}
\end{equation}
where
\ben
R_1(A) := [\Int_{\g_1(A)} - \Int_{\g_1(-A)}]\left(u^{\text{t}}(x;z_1)\frac{\pa u^{\text{t}}}{\pa x_2}(x;z_2)
-u^{\text{t}}(x;z_2)\frac{\pa u^{\text{t}}}{\pa x_2}(x;z_1)\right)\dd s
\enn

Similarly, applying Green's first theorem to $u^{\text{t}}(\cdot;z_1)$ and $u^{\text{t}}(\cdot;z_2)$
in $\Om_{2,h}(A)\backslash B_{\varepsilon}(z_2)$ gives
\begin{equation}\label{eq56}
\begin{aligned}
0 = \Int_{\Om_{2,h}(A)}\left(u^{\text{t}}(x;z_1)\Delta u^{\text{t}}(x;z_2)
- u^{\text{t}}(x;z_2)\Delta u^{\text{t}}(x;z_1)\right) \dd x = \\
\Int_{\G_{-h}(A)}\left(u^{\text{t}}(x;z_1)\frac{\pa u^{\text{t}}}{\pa \nu}(x;z_2)
- u^{\text{t}}(x;z_2)\frac{\pa u^{\text{t}}}{\pa \nu}(x;z_1)\right)\dd s + R_2(A) + \\
\Int_{\G(A)}\left(u^{\text{t}}(x;z_1)|^- \frac{\pa u^{\text{t}}}{\pa \nu}(x;z_2)|^-
- u^{\text{t}}(x;z_2)|^-\frac{\pa u^{\text{t}}}{\pa \nu}(x;z_1)|^-\right) \dd s - \\
\Int_{\pa B_{\varepsilon}(z_2)}\left(u^{\text{t}}(x;z_1)\frac{\pa u^{\text{t}}}{\pa \nu}(x;z_2)
- u^{\text{t}}(x;z_2)\frac{\pa u^{\text{t}}}{\pa \nu}(x;z_1)\right)\dd s - \\
\Int_{\pa D}\left(u^{\text{t}}(x;z_1)\frac{\pa u^{\text{t}}}{\pa \nu}(x;z_2)
- u^{\text{t}}(x;z_2)\frac{\pa u^{\text{t}}}{\pa \nu}(x;z_1)\right)\dd s
\end{aligned}
\end{equation}
where
\ben
R_2(A) := [\Int_{\g_2(A)} - \Int_{\g_2(-A)}]\left(u^{\text{t}}(x;z_1)\frac{\pa u^{\text{t}}}{\pa x_2}(x;z_2)
- u^{\text{t}}(x;z_2)\frac{\pa u^{\text{t}}}{\pa x_2}(x;z_1)\right)\dd s
\enn
Since $u^{\text{t}}(\cdot;z_1)$ and $u^{\text{t}}(\cdot;z_2)$ satisfy the Helmholtz equation in
$\Om_{2,h}(A)\backslash B_{\varepsilon}(z_2)$, the transmission conditions on $\G$ and the boundary
conditions on $\pa D$, \eqref{eq55} and \eqref{eq56} yields
\begin{equation}\label{eq61}
\begin{aligned}
0 = \Int_{\G_{h}(A)}\left(u^{\text{t}}(x;z_1)\frac{\pa u^{\text{t}}(x;z_2)}{\pa \nu(x)}
- u^{\text{t}}(x;z_2)\frac{\pa u^{\text{t}}(x;z_1)}{\pa \nu(x)}\right)\dd s + \\
\Int_{\G_{-h}(A)}\left(u^{\text{t}}(x;z_1)\frac{\pa u^{\text{t}}}{\pa \nu}(x;z_2)
- u^{\text{t}}(x;z_2)\frac{\pa u^{\text{t}}}{\pa \nu}(x;z_1)\right)\dd s + \\
R_1(A) -  \Int_{\pa B_{\varepsilon}(z_1)}\left(u^{\text{t}}(x;z_1)\frac{\pa u^{\text{t}}}{\pa \nu}(x;z_2)
- u^{\text{t}}(x;z_2)\frac{\pa u^{\text{t}}}{\pa \nu}(x;z_1)\right)\dd s + \\
R_2(A) -\Int_{\pa B_{\varepsilon}(z_2)}\left(u^{\text{t}}(x;z_1)\frac{\pa u^{\text{t}}}{\pa \nu}(x;z_2)
- u^{\text{t}}(x;z_2)\frac{\pa u^{\text{t}}}{\pa \nu}(x;z_1) \right)\dd s.
\end{aligned}
\end{equation}
As $\varepsilon \rightarrow 0$, we have
\begin{equation*}
\begin{aligned}
\Int_{\pa B_{\varepsilon}(z_1)}\left(u^{\text{t}}(x;z_1)\frac{\pa u^{\text{t}}}{\pa \nu}(x;z_2)
- u^{\text{t}}(x;z_2)\frac{\pa u^{\text{t}}}{\pa \nu}(x;z_1)\right)\dd s = \\
\Int_{\pa B_{\varepsilon}(z_1)}\left(G_{k_1}(x;z_1)\frac{\pa u^{\text{t}}}{\pa \nu}(x;z_2)
- u^{\text{t}}(x;z_2)\frac{\pa G_{k_1}}{\pa \nu}(x;z_1)\right)\dd s + \\
\Int_{\pa B_{\varepsilon}(z_1)}\left(u(x;z_1)\frac{\pa u^{\text{t}}}{\pa \nu}(x;z_2)
- u^{\text{t}}(x;z_2)\frac{\pa u}{\pa \nu}(x;z_1)\right)\dd s\rightarrow u^{\text{t}}(z_1;z_2)
\end{aligned}
\end{equation*}
where we have used Theorem 2.1 in \cite{Colton}.

Similarly,
\begin{equation*}
\begin{aligned}
\Int_{\pa B_{\varepsilon}(z_2)}\left(u^{\text{t}}(x;z_1)\frac{\pa u^{\text{t}}}{\pa \nu}(x;z_2)
- u^{\text{t}}(x;z_2)\frac{\pa u^{\text{t}}}{\pa \nu}(x;z_1)\right)\dd s
\rightarrow -u^{\text{t}}(z_2;z_1) \quad \text{as}\;\;\varepsilon \rightarrow 0.
\end{aligned}
\end{equation*}
Since $u^{\text{t}}(\cdot;z_j)\in \widetilde{V_h},\ j=1,2$ and satisfy the angular-spectrum
representation \eqref{eq4} and \eqref{eq5}, and by \eqref{eq10} in Lemma \ref{lem0}, we have
\ben
&&\Int_{\G_{h}(A)}\left(u^{\text{t}}(x;z_1)\frac{\pa u^{\text{t}}}{\pa \nu}(x;z_2)
- u^{\text{t}}(x;z_2)\frac{\pa u^{\text{t}}}{\pa \nu}(x;z_1)\right)\dd s\\
&&\qquad\qquad=\Int_{\G_{-h}(A)}\left(u^{\text{t}}(x;z_1)\frac{\pa u^{\text{t}}}{\pa \nu}(x;z_2)
- u^{\text{t}}(x;z_2)\frac{\pa u^{\text{t}}}{\pa \nu}(x;z_1)\right)\dd s\rightarrow 0
\enn
as $A \rightarrow +\infty$. Noting that $u^{\text{t}}(\cdot,z_i) \in \widetilde{V_h}$,
it follows from the asymptotic properties of $G_{k_1}(\cdot,z_i)\ i = 1,2$ and their derivatives that
$R_j(A)\rightarrow 0 \ \text{as}\ A\rightarrow +\infty, \ j = 1,2.$
Thus, $u^{\text{t}}(z_2;z_1) = u^{\text{t}}(z_1;z_2)$ follows from \eqref{eq61} by letting $\varepsilon \rightarrow 0$
and $A \rightarrow +\infty$. The proof is complete.
\end{proof}

\begin{remark}
By the symmetry of $G_{k}(x,z)$, we also have the reciprocity relation of the scattered field.
\be\label{eq_rr}
u(z_1;z_2) = u(z_2;z_1),\quad z_1,z_2\in\Om_1\;\;\text{or}\;\;\Om_2\ba\ov{D}\;\;\text{and}\;\;z_1\neq z_2
\en
\end{remark}

Suppose that $\G$ and $\widetilde{\G}$ are two rough interfaces and that $D$ and $\widetilde{D}$ are two
impenetrable obstacles with the boundary physical property $\mathcal{B}$ and $\widetilde{\mathcal{B}}$ respectively.
Define $\widetilde{u}(\cdot;z)$ and $\widetilde{u}^{\text{t}}(\cdot;z)$ to be the scattered and total field
due to the PSW and given by the scattering problem \SP with $\widetilde{\G},\widetilde{\Om_1},\widetilde{\Om_2},\widetilde{D},\widetilde{\mathcal{B}}$.
The fields $u',u'^{\text{t}},\widetilde{u}', \widetilde{u}'^{\text{t}}$ due to the HSPSW can be defined accordingly.

We now have the uniqueness result for the inverse scattering problem.

\begin{theorem}\label{thm5_1}
If the scattered field $u(x;z)=\widetilde{u}(x;z)$ for all $z\in\Sigma_s\subset\G_b$
and $x\in\Sigma_r\subset\G_c$, then $\G=\widetilde{\G}, D=\widetilde{D},\mathcal{B}=\widetilde{\mathcal{B}}$.
\end{theorem}

\begin{proof}
{\bf Step 1.} We prove that $\G = \widetilde{\G}$.

Let $\Om$ be the unbounded connected component of $\Om_1\cap \widetilde{\Om_1}$. For $z \in \Om$, we first claim that
\be\label{eq5_1}
u(x;z) = \widetilde{u}(x;z) \quad \text{for all} \ x\in \overline{\Om}
\en
Since $u(\cdot;z)$ and $\widetilde{u}(\cdot;z)$ are both analytic in $\Om$ and $u(x;z) = \widetilde{u}(x;z)$ for 
all $x\in\Sigma_r$, then $u(x;z)=\widetilde{u}(x;z),x\in \G_c$. From the uniqueness of the Dirichlet problem in $U^+_c$, 
we know that \eqref{eq5_1} holds for $x\in U^+_c, z\in\Sigma_s$. By the unique continuation principle, \eqref{eq5_1} also 
holds for $x\in\Om,$ $z\in\Sigma_s$. By the reciprocity relation \eqref{eq_rr}, we have $u(z;x)=u(x;z)$ for $z\in\Sigma_s,$
$x\in\Om$. Repeating the above argument, we obtain that $u(x;z)=\widetilde{u}(x;z)$ for all $z\in\Om,$ $x\in\Om$. 
Since the scattered fields are continuous up to the boundary, \eqref{eq5_1} holds. By Theorem \ref{thm3} and \eqref{eq5_1} 
we have
\ben
u^{\prime}(x;z)=\widetilde{u}^{\prime}(x;z)\quad\text{for all}\;\;z\in\Om,\;\;x\in\overline{\Om}.
\enn
Assume that $\G\neq\widetilde{\G}$. Without loss of generality, we may assume that there exists $z^{\ast}\in\G\ba\widetilde{\G}$. 
Define $z_j:=z^{\ast}+(\delta/j)\nu(z^{\ast}), j\in\mathbb{N}^+,$ with $\delta>0$ such that $z_j\in B_{\delta}(z^{\ast})$, 
where $B_{\delta}(z^{\ast})$ is a ball centred at $z^\ast$ and with radius $\delta$ satisfying that 
$\ov{B_{2\delta}(z^{\ast})}\subset\widetilde{\Om_1}$. Choose a small domain $\Om_0\subset\Om_2$ with a $C^2$-boundary $\pa\Om_0$ 
such that $B_{2\delta}\cap\Om_2\subset\Om_0\subset\overline{\Om_0}\subset\widetilde{\Om_1}$ and let 
$d:=\text{dist}(\Om_0, \widetilde{\G})>0$. See the geometric setting in Figure \ref{fig2}. Define the scattered field 
$u^{\prime}_j(x):=u^{\prime}(x;z_j), \widetilde{u}^{\prime}_j(x):=\widetilde{u}^{\prime}(x;z_j)$ and the total field 
$u^{\prime\text{t}}_j(x):=u^{\prime\text{t}}(x;z_j),\widetilde{u}^{\prime\text{t}}_j(x):=\widetilde{u}^{\prime \text{t}}(x;z_j)$. 
Also we set $V_j=\widetilde{u}^{\prime\text{t}}_j|_{\Om_0}$ and $U_j=u^{\prime}_j|_{\Om_0}$ in $\Om_0$. 
Then $V_j$ and $U_j$ satisfy \ITP in $\Om_0$ with the boundary data 
$f_{1,j}:=(u^{\prime}_j-\widetilde{u}^{\prime\text{t}}_j)|_{\pa\Om_0}$, 
$f_{2,j}:=\pa(u^{\prime}_j-\widetilde{u}^{\prime\text{t}}_j)/\pa\nu|_{\pa\Om_0}$, $k^2=k^2_1, nk^2=k^2_2$ and $\im(n)\geq 0$. 
It is clear that $f_{1,j}=f_{2,j}=0$ on $\G^{\ast}=\G\cap\pa\Om_0$. Since $z^{\ast}$ has a positive distance 
from $\widetilde{\G}$, it follows from \ref{rem4_1} (ii) that
\ben
\|\widetilde{u}^{\prime}_j\|_{L^2(\Om_0)}+\|\widetilde{u}^{\prime}_j\|_{H^{\frac{1}{2}}(\pa\Om_0\ba\G^{\ast})} 
+\|\frac{\pa\widetilde{u}^{\prime}_j}{\pa\nu}\|_{H^{-\frac{1}{2}}(\pa\Om_0\ba\G^{\ast})}\leq C
\enn
uniformly with respect to $j\in\mathbb{N}$. Let $K=\Om_0\ba B_{2\delta}(z^{\ast})$. 
Then $\text{dist}(K,B_{\delta}(z^{\ast}))=\delta$. Theorem \ref{thm4_1} implies that $\|u^{\prime}_j\|_{H^1(K)}\leq C$
uniformly with respect to $j\in\mathbb{N}$. This, together with the trace theorem, implies that
\ben
\|u^{\prime}_j\|_{H^{\frac{1}{2}}(\pa\Om_0\ba\G^{\ast})}
+\|\frac{\pa u^{\prime}_j}{\pa\nu}\|_{H^{-\frac{1}{2}}(\pa\Om_0\ba\G^{\ast})}\leq C
\enn
uniformly with respect to $j\in\mathbb{N}$. Now we can invoke Theorem \ref{thm4} to conclude that
\be\label{eq:total}
\|\widetilde{u}^{\prime \text{t}}_j\|_{L^2(\Om_0)}\leq C
\en 
uniformly with respect to $j\in\N$. In fact, if $\im(n)>0$, then, by Theorem \ref{thm4} (i) 
the constructed interior transmission problem on $\Omega_0$ is well-posed. On the other hand, if $\im(n)=0$, 
then, by Theorem \ref{thm4} (ii) we can choose $\Omega_0$ sufficiently small so that $k^2$ is not an interior 
transmission eigenvalue on $\Omega_0$. In either case, \eqref{eq:total} follows from the estimate \eqref{eq:stability}.

Notice that the scattered field $\widetilde{u}'_j$ is bounded uniformly with respect to $j\in \N$, so we get the estimate
$\|G'_{k_1}(\cdot,z_j)\|_{L^2(\Om_0)}\leq C$
uniformly with respect to $j\in\N$. This is a contradiction since $G'_{k_1}(\cdot,z^{\ast})$ is not locally 
integrable in $\Om_0$. Thus we have $\G=\widetilde{\G}$.

\begin{figure}
\centering
\includegraphics[scale=0.6]{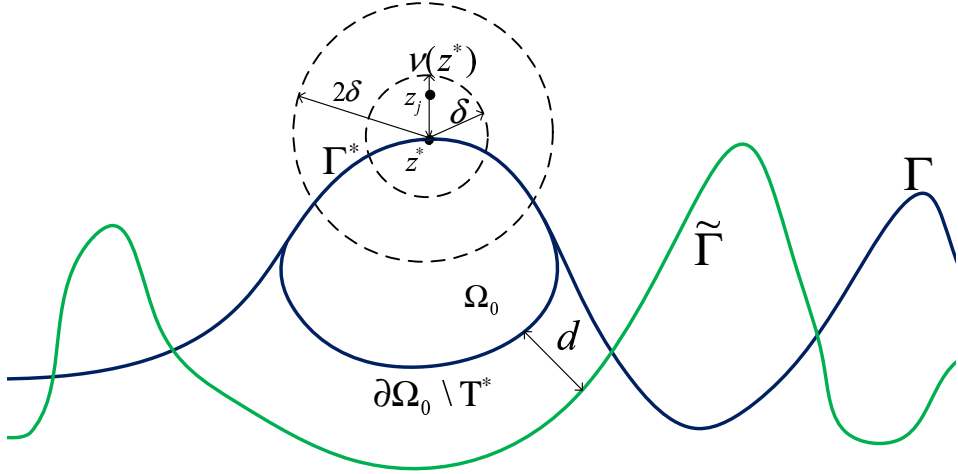}
\caption{Geometry in Step 1.}\label{fig2}
\end{figure}

{\bf Step 2.} We show that $D=\widetilde{D}$.

By Step 1, we already have $\Om_1=\widetilde{\Om_1}$. Let $S$ be the unbounded connected component of
$\R^2\backslash\{\overline{\Om_1}\cup\overline{D\cup\widetilde{D}}\}$. For $z\in S$, we claim that
\be\label{eq5_2}
u(x;z)=\widetilde{u}(x;z)\quad \text{for all}\;\; x\in\overline{S}.
\en
In fact, for $z\in\Om_1$, by Step 1 we have $u^{\text{t}}(x;z)=\widetilde{u}^{\text{t}}(x;z)$
for all $x\in\overline{\Om_1}$ and $x\neq z$, and by the transmission conditions we obtain that
\ben
u(x;z)=\widetilde{u}(x;z),\;\;\frac{\pa u}{\pa\nu}(x;z)=\frac{\pa\widetilde{u}}{\pa\nu}(x;z)\;\;\text{for all}\quad x\in\G,z\in\Om_1
\enn
Since $u$ and $\widetilde{u}$ satisfy the Hemholtz equation in $S$, Holmgren's uniqueness theorem implies that
\ben
u(x;z) = \widetilde{u}(x;z) \quad \text{for all}\;\; x\in \overline{S}, z\in\Om_1
\enn
This together with the reciprocity relation of the total field, implies that $u(z;x)=\widetilde{u}(z;x)$ 
for all $x\in S,z\in\Om_1$. Regarding $u$ and $\widetilde{u}$ as functions of $z$ and repeating the same argument as above 
yield that $u(z;x) = \widetilde{u}(z;x)\quad\text{for all}\;\;x,z\in S.$
Since the scattered fields are continuous up to the boundary, by exchanging $z$ and $x$, \eqref{eq5_2} holds.

Assume that $D\neq\widetilde{D}$. Without loss of generality, we may assume that there exists $z^{\ast}\in\pa D\ba\pa\wi{D}$. 
Define $z_j:=z^{\ast}+(\delta/j)\nu(z^{\ast}), j\in\mathbb{N}^+,$ with $\delta>0$
such that $z_j\in B_{\delta}(z^{\ast})$ and $\overline{B_{\delta}(z^{\ast})}\cap\widetilde{D}=\varnothing$.
Since there is a positive distance between $B_{\delta}(z^{\ast})$ and $\widetilde{D}$, by \eqref{eq5_2} and
Remark \ref{rem4_1} (ii), it follows that
\ben
\|\frac{\pa u(\cdot;z_j)}{\pa \nu} + i\beta u(\cdot;z_j)\|_{H^{-\frac{1}{2}}(\pa \G_2 \cup B_{\delta}(z^{\ast}))}
+\|u(\cdot;z_j)\|_{H^{\frac{1}{2}}(\pa \G_1 \cup B_{\delta}(z^{\ast}))}\leq C,
\enn
where the constant $C > 0$ is independent of $j$. But $u(\cdot;z_j)$ satisfies boundary conditions, so
\ben
&&\|\frac{\pa u(\cdot;z_j)}{\pa\nu} + i\beta u(\cdot;z_j)\|_{H^{-\frac{1}{2}}(\pa\G_2\cup B_{\delta}(z^{\ast}))}
+\|u(\cdot;z_j)\|_{H^{\frac{1}{2}}(\pa\G_1\cup B_{\delta}(z^{\ast}))}\\
&&\quad=\|\frac{\pa G_{k_2}(\cdot;z_j)}{\pa\nu}+i\beta G_{k_2}(\cdot;z_j)\|_{H^{-\frac{1}{2}}(\pa\G_2\cup B_{\delta}(z^{\ast}))}
+\|G_{k_2}(\cdot;z_j)\|_{H^{\frac{1}{2}}(\pa\G_1\cup B_{\delta}(z^{\ast}))}
\rightarrow\infty
\enn
as $j\rightarrow\infty.$ This is a contradiction, which means that $D=\widetilde{D}$.

{\bf Step 3.} We show that the physical property is uniquely determined, that is, $\mathcal{B}=\wi{\mathcal{B}}$. 
First, as a result of Step 2, we claim that
\be \label{eq5_3}
\G_i = \widetilde{\G}_i,\quad i = 1,2
\en
In fact, suppose \eqref{eq5_3} is not true. Then $\G_1\cap\widetilde{\G}_2\neq\varnothing$. For $z\in\Sigma_s$ 
we have $u(\cdot;z)=\pa u(\cdot;z)/\pa\nu = 0$ on $\G_1\cap\widetilde{\G}_2$, and by Holmgren's uniqueness theorem,
$u(\cdot;z)=0$ in $\Om_2$. Thus, $u^{\text{t}}(\cdot;z)=\pa u^{\text{t}}(\cdot;z)/\pa \nu=0$  on $\G$.
Applying Holmgren's uniqueness theorem again, we have $u(\cdot;z) = G_{k_1}(\cdot;z)$ in $\Om_1\ba B_{\delta}(z)$
for any $\delta >0$ such that $\overline{B_{\delta}(z)}\cap\G=\varnothing$. Let $\delta\rightarrow 0$ to get that $\|u(x;z)\|_{H^1(B_{\delta}(z))}\rightarrow\infty$, which contradicts to Remark \ref{rem4_1} (ii).
Thus, \eqref{eq5_3} holds.

Next we may assume that $\G_1$ and $\widetilde{\G}_1$ are both nonempty. If the impedance
function $\beta\neq\widetilde{\beta}$, then from the boundary condition on $\G_1$
\ben
\frac{\pa u(\cdot;z)}{\pa\nu}+i\beta u(\cdot;z)=0,\;\; \frac{\pa u(\cdot;z)}{\pa\nu}
 + i\widetilde{\beta}u(\cdot;z)=0\quad\text{on}\;\; \G_1,\;\;\text{for}\;\;z\in\Sigma_s,
\enn
which gives
\ben
(\beta - \widetilde{\beta})u(\cdot;z)=0\quad\text{on}\;\;\G_1\;\;\text{for}\;\; z\in\Sigma_s.
\enn
Consequently, $\pa u(\cdot;z)/\pa\nu=u(\cdot;z)=0$ on the open set $\{x\in\pa D:\;\beta(x)\neq\wi{\beta}(x)\}$. 
Then, we get the same contradiction as that in proving \eqref{eq5_3}.
Hence, $\mathcal{B}=\widetilde{\mathcal{B}}$. The proof is thus finished.
\end{proof}

\section*{Acknowledgements}

Most sections of this paper were finished while the first author (YL) was studying at AMSS, Chinese Academy of Sciences.
The work was partly supported by the NNSF of China under grants 61379093, 91430102 and 11501558.


\begin{thebibliography}{99}



\bibitem{Bao} G. Bao, A uniqueness theorem for an inverse problem in periodic diffractive optics,
{\em Inverse Problems \bf10} (1994), 335-340.



\bibitem{FC2} F. Cakoni, D. Gintides and H. Haddar, {\em The existence of an infinite discrete set of
transmission eigenvalues}, SIAM J. Math. Anal. 42 (2010), 237-255.


\bibitem{CCC08}{F. Cakoni, M. Cayoren and D. Colton}, {\em Transmission eigenvalues and the nondestructive testing
of dielectrics}, Inverse Problems, 26 (2008) 065016.



\bibitem{SN7} S.N. Chandler-Wilde and C. Ross, {\em  Uniqueness results for direct and inverse scattering
by infinite surfaces in a lossy medium}, Inverse Problems, 10 (1995), 1063-1067.

\bibitem{SN4} S.N. Chandler-Wilde, E. Heinemeyer, and R. Potthast, {\em A well-posed integral equation
formulation for three-dimensional rough surface scattering}, Proc. R. Soc. London, A462 (2006), 3683-3705.

\bibitem{SN6} S.N. Chandler-Wilde, J. Elschner, {\em Variational approach in weighed Sobolev spaces to
scattering by unbounded rough surfaces}, SIAM J. Math. Anal., 42 (2010), 2554-2580.

\bibitem{SN3} S.N. Chandler-Wilde and P. Monk, {\em Existence, uniqueness, and variational methods for
scattering by unbounded rough surfaces}, SIAM J. Math. Anal., 37 (2005), 598-618.

\bibitem{ZhangRS} S.N. Chandler-Wilde and Bo Zhang, {\em Electromagnetic scattering by an inhomogeneous conducting 
or dielectric layer on a perfectly conducting plate}, Proc. R. Soc. London, A454 (1998), 519-542.

\bibitem{CZ1999} S.N. Chandler-Wilde and B. Zhang, {\em Scattering of electromagnetic
waves by rough surfaces and inhomogeneous layers}, SIAM J. Math. Anal., 30 (1999), 559--583.

\bibitem{SN5} S.N. Chandler-Wilde and B. Zhang, {\em A uniqueness result for scattering by infinite rough surfaces},
SIAM J. Appl. Math., 58 (1998), 1774-1790.

\bibitem{Colton} D. Colton and R. Kress, {\em Inverse Acoustic and Electromagnetic Scattering Theory} 3nd ed, 
Springer-Verlag, Berlin, 2013.

\bibitem{CPS07}{D. Colton, L. P\"a\.iv\"arinta and J. Sylvester}, {\em The interior transmission problem},
Inverse Probl. Imaging, 1 (2007), 13-28.

\bibitem{Delbary} F. Delbary, K. Erhard, R. Kress, R. Potthast and J. Schulz, {\em Inverse electromagnetic scattering in
a two-layered medium with an application to mine detection}, Inverse Problems, 24(2008), 015002.

\bibitem{DeS} J.A. DeSanto, {\em Scattering by rough surfaces}, in: Scattering: Scattering and Inverse Scattering
in Pure and Applied Science, R. Pike and P. Sabatier, eds., Academic Press, New York, 2002, 15-36.

\bibitem{HE} J. Elschner and G. Hu, {\em Inverse scattering of electromagnetic waves by multilayered structures
Uniqueness in TM mode}, Inverse Problem and Imaging, 15 (2011), 1565-1587.

\bibitem{Gilbarg} D. Gilbarg and N.S. Trudinger, {\em Elliptic Partial Differential Equations of Second Order},
2nd ed, Springer, Berlin, 1983.

\bibitem{Griesmaier} R. Griesmaier, {\em An asymptotic factorization method for inverse electromagnetic scattering
in layered media}, SIAM J. Appl. Math., 68 (2008), 1378-1403.

\bibitem{HL2011} H. Haddar and A. Lechleiter, {\em Electromagnetic wave scattering from rough
penetrable layers}, SIAM J. Math. Anal., 43 (2011), 2418--2443.

\bibitem{Kirsch2} F. Hettlich and A. Kirsch, {\em Schiffer's theorem in inverse scattering for periodic structures},
Inverse Problems, 13 (1997), 351-361.

\bibitem{Hu} G. Hu, {\em Inverse wave scattering by unbounded obstacles: uniqueness for the two dimensional
Helmholtz equation}, Appl. Anal, 91 (2012), 703-717.

\bibitem{Hu15} G. Hu, X. Liu, F. Qu and B. Zhang, {\em Variational approach to scattering by unbounded rough surfaces 
with Neumann and generalized impedance boundary conditions}, Commun. Math. Sci., 13 (2015), 511-537.

\bibitem{Ih} F. Ihlenburg, {\em Finite Element Analysis of Acoustic Scattering}, Springer, Berlin, 1998

\bibitem{Isakov} V. Isakov, {\em On uniqueness in the inverse transmission scattering problem}, Comm.
Part. Diff. Equat., 15 (1990), 1565-1587.

\bibitem{Kirsch3} A. Kirsch and R. Kress, {\em Uniqueness in inverse obstacle scattering}, 
Inverse Problems, 9 (1993), 285-299.

\bibitem{Kirsch1} A. Kirsch, {\em Uniqueness theorems in inverse scattering theory for periodic structures},
Inverse Problems, 10 (1994), 145-152.

\bibitem{AS} A. Lechleiter and S. Ritterbusch, {\em A variational method for wave scattering from penetrable
rough layers}, IMA J. Appl. Math., 75 (2010), 366-391.

\bibitem{Li} P. Li, H. Wu and W. Zheng, {\em Electromagnetic scattering by unbounded rough surfaces}, SIAM
J. Math. Anal., 43 (2011), 1205-1231.


\bibitem{Li16} J. Li, G. Sun and R. Zhang, {\em The numerical solution of scattering by infinite rough interfaces 
based on the integral equation method}, Comput. Math. Appl., 71 (2016), 1491-1502.



\bibitem{LZ1} X. Liu and B. Zhang, {\em Direct and inverse obstacle scattering problems in a piecewise
homogeneous medium}, SIAM J. Appl. Math., 70 (2010), 3105-3120.

\bibitem{LZ2} X. Liu and B. Zhang, {\em A uniqueness result for the inverse electromagnetic
scattering problem in a two-layered medium}, Inverse Problems, 26 (2010) 105007 (11pp).

\bibitem{NAC} D. Natroshvili, T. Arens and S.N. Chandler-Wilde, {\em Uniqueness, existence, and integral
equation formulations for interface scattering problems}, Memoirs on Differential Equations and
Mathematical Physics, 30 (2003), 105-146.

\bibitem{Ogilvy} J.A. Ogilvy, {\em Theory of Wave Scattering from Random Rough Surfaces}, Adam Hilger, Bristol,
UK, 1991.


\bibitem{S11}{J. Sun}, {\em Estimation of transmission eigenvalues and the index of refraction from Cauchy data},
Inverse Problems, 27 (2011), 015009.

\bibitem{AG} A.G. Voronovich, {\em Wave Scattering from Rough Surfaces}, 2nd ed., Springer, Berlin, 1998.

\bibitem{Chew} K. Warnick and W.C. Chew, {\em Numerical simulation methods for rough surface scattering},
Waves Random Media, 11 (2001), R1-R30.

\bibitem{YZ1} J. Yang and B. Zhang, {\em Uniqueness results in the inverse scattering problem for periodic structures},
Math. Methods Appl. Sci, 35 (2012), 828-838.

\bibitem{YZ2} J. Yang and B. Zhang, {\em Inverse electromagnetic scattering problems by a doubly periodic structure},
Methods Appl. Anal., 18 (2011), 111-126. 

\bibitem{YZZ} J. Yang, B. Zhang and H. Zhang, {\em Uniqueness in inverse acoustic and electromagnetic scattering
by penetrable obstacles}, arXiv:1305.0917v2, 2013.

\bibitem{SN1} B. Zhang and S.N. Chandler-Wilde, {\em Acoustic scattering by an inhomogeneous layer on a rigid plate},
SIAM. J. Appl. Math, 58 (1998), 1931-1950.

\bibitem{SN2} B. Zhang and S.N. Chandler-Wilde, {\em Integral equation methods for scattering by infinite
rough surfaces}, Math. Methods Appl. Sci, 26 (2003), 463-488.

\end{thebibliography}
\end{document}